\documentclass[
	english
]{scrartcl}

\usepackage[T1]{fontenc}
\usepackage[utf8]{inputenc}

\usepackage{mathtools}
\usepackage{amsmath,amsfonts,amsthm,amssymb}
\usepackage{bm}
\usepackage[colorlinks,citecolor=blue]{hyperref}
\usepackage{doi,natbib}
\usepackage{cancel}

\usepackage{algorithm,algpseudocode}
\usepackage[labelformat=simple]{subcaption}

\usepackage{changes} 

\usepackage{enumitem}
\setlist[enumerate]{label=(\alph*)}

\usepackage{preamble}

\usepackage{marginnote}

\usepackage[capitalise,nameinlink]{cleveref}

\crefname{figure}{Figure}{Figures}

\allowdisplaybreaks

\numberwithin{equation}{section}
\numberwithin{figure}{section}

\newcommand\norm[1]{\left\Vert#1\right\Vert}
\newcommand\abs[1]{\left\vert#1\right\vert}

\newcommand\N{\mathbb{N}}
\newcommand\R{\mathbb{R}}

\newcommand{\LL}{\mathcal L}

\newcommand{\intr}{\operatorname{int}}

\newcommand{\dist}{\operatorname{dist}}
\newcommand{\conv}{\operatorname{conv}}

\newcommand{\dom}{\operatorname{dom}}
\newcommand{\gph}{\operatorname{gph}}
\newcommand{\epi}{\operatorname{epi}}
\newcommand{\clconv}{\operatorname{\overline{\conv}}}

\newcommand{\tto}{\rightrightarrows}
\renewcommand{\Im}{\operatorname{Im}}
\newcommand{\wStarlimsup}{\operatorname{w}^*-\limsup}

\newcommand{\weakly}{\rightharpoonup}
\newcommand{\weaklystar}{\stackrel\ast\rightharpoonup}

\DeclareMathAlphabet{\mathpzc}{OT1}{pzc}{m}{it}

\newtheorem{theorem}{Theorem}[section]
\newtheorem{lemma}[theorem]{Lemma}
\newtheorem{proposition}[theorem]{Proposition}
\newtheorem{assumption}[theorem]{Assumption}
\newtheorem{corollary}[theorem]{Corollary}
\newtheorem{remark}[theorem]{Remark}
\newtheorem{definition}[theorem]{Definition}
\newtheorem{example}[theorem]{Example}

\definecolor{mygreen}{rgb}{0.0,0.7,0.0}

\newcommand{\eps}{\varepsilon}

\newcommand{\ga}{\gamma}
\newcommand{\de}{\delta}
\newcommand {\B} {\mathbb B}
\newcommand{\bx}{\bar x}
\newcommand{\by}{\bar y}
\newcommand{\bz}{\bar z}
\newcommand {\sd} {\partial}
\newcommand{\bsd}{\overline{\partial}}
\newcommand{\ang}[1]{\left\langle #1 \right\rangle}

\begin{document}

\title{
	Optimality conditions, approximate stationarity, and applications\\
	-- a story beyond Lipschitzness
}
\author{%
	Alexander Y.\ Kruger%
	\footnote{%
		Federation University Australia,
		Centre for Informatics and Applied Optimization,
		School of Engineering, Information Technology and Physical Sciences,
		Ballarat VIC 3353,
		Australia,
		\email{a.kruger@federation.edu.au},
		\url{https://asterius.federation.edu.au/akruger/},
		\orcid{0000-0002-7861-7380}
	}
	\and
	Patrick Mehlitz%
	\footnote{%
		Brandenburgische Technische Universit\"at Cottbus--Senftenberg,
		Institute of Mathematics,
		03046 Cottbus,
		Germany,
		\email{mehlitz@b-tu.de},
		\url{https://www.b-tu.de/fg-optimale-steuerung/team/dr-patrick-mehlitz},
		\orcid{0000-0002-9355-850X}
		}
	}

\publishers{}
\maketitle

\begin{abstract}
	Approximate necessary optimality conditions in terms of Fr\'{e}chet subgradients and normals
	for a rather general optimization problem with a potentially non-Lipschitzian objective function
	are established with the aid of Ekeland's variational principle,
	the fuzzy Fr\'{e}chet subdifferential sum rule,
	and a novel notion of lower semicontinuity relative to a set-valued mapping or set.
	Feasible points satisfying these optimality conditions are referred to as approximately stationary.
	As applications, we derive a new general version of the extremal principle.
	Furthermore, we study approximate stationarity conditions for an optimization problem with a composite objective function and geometric constraints,
	a qualification condition guaranteeing that approximately stationary points of such a problem are M-stationary, and
	a multiplier-penalty-method which naturally computes approximately stationary points of the underlying problem.
	Finally, necessary optimality conditions for an optimal control problem with a non-Lipschitzian sparsity-promoting term
	in the objective function are established.
\end{abstract}


\begin{msc}	
	\mscLink{49J52} \mscLink{49J53}, \mscLink{49K27}, \mscLink{90C30}, \mscLink{90C48}
\end{msc}

\section{Introduction}\label{sec:introduction}

Approximate stationarity conditions, claiming that, along a convergent sequence, a classical
stationarity condition (like a multiplier rule) holds up to a tolerance which tends to zero,
have proved
to be a powerful tool in mathematical optimization throughout the last decades.
The particular interest in such conditions is based on two prominent features.
First, they often serve as
necessary optimality conditions
even in the absence of
constraint qualifications. Second, different classes of solution algorithms for the
computational treatment of optimization problems naturally produce sequences whose accumulation
points are approximately stationary.
Approximate stationarity conditions can be traced back to the early 1980s,
see \cite{KruMor80,Kruger1985}, where they popped up
as a consequence of the famous \emph{extremal principle}.
The latter geometric result, when formulated in infinite dimensions in terms of Fr\'{e}chet normals,
can itself be interpreted as a kind of approximate stationarity, see \cite{KruMor80,Kru03,Mordukhovich2006}.
In \cite{AndreaniMartinezSvaiter2010,AndreaniHaeserMartinez2011}, this fundamental concept, which is referred
to as \emph{Approximate Karush--Kuhn--Tucker (AKKT) stationarity} in these papers, has been
rediscovered due to its significant relevance in the context of numerical standard nonlinear
programming. A notable feature of AKKT-stationary points is the potential unboundedness of
the associated sequence of Lagrange-multipliers. The latter already depicts that AKKT-stationary
points do not need to satisfy the classical KKT conditions. This observation gave rise to the
investigation of conditions ensuring that AKKT-stationary points actually are KKT points,
see e.g.\ \cite{AndreaniMartinezRamosSilva2016}. The resulting constraint qualifications for
the underlying nonlinear optimization problem turned out to be comparatively weak.
During the last decade, reasonable notions of approximate stationarity have been introduced for more
challenging classes of optimization problems like programs with complementarity, see \cite{AndreaniHaeserSecchinSilva2019,Ramos2019},
cardinality, see \cite{KanzowRaharjaSchwartz2021}, conic, see \cite{AndreaniHaeserViana2020},
nonsmooth, see \cite{HelouSantosSimoes2020,Mehlitz2020b,Mehlitz2021}, and geometric constraints, see \cite{JiaKanzowMehlitzWachsmuth2021},
in the finite-dimensional situation.
A generalization to optimization problems in abstract Banach spaces can be found in \cite{BoergensKanzowMehlitzWachsmuth2019}.
In all these papers, the underlying optimization problem's objective function is assumed to be locally Lipschitzian.
Note that the (local) Lipschitz property of the (all but one) functions involved is a key assumption in most conventional
subdifferential calculus results in infinite dimensions in convex and nonconvex settings, see e.g.\ the sum rules in \cref{lem:SR}.
However, as several prominent applications like sparse portfolio selection, compressed sensing, edge-preserving image
restoration, low-rank matrix completion, or signal processing, where the objective function is often only
lower semicontinuous, demonstrate, Lipschitz continuity
might be a restrictive property of the data.
The purpose of this paper is to provide a reasonable extension of approximate stationarity to a
rather general class of optimization problems in Banach spaces with a lower semicontinuous objective function and
generalized equation constraints generated by a set-valued mapping in order to open the topic up to the aforementioned
challenging applications.

Our general approach to a notion of approximate stationarity,
which serves as a necessary optimality condition,
is based on two major classical tools: \emph{Ekeland's variational principle}, see \cite{Ekeland1974}, and the \emph{fuzzy calculus} of Fr\'{e}chet normals,
see \cite{Ioffe2017,Kru03}.
Another convenient ingredient of the theory is a new notion of lower semicontinuity of extended-real-valued functions relative
to a given set-valued mapping which holds for free in finite dimensions.
We illustrate our findings in the context of generalized set separation and derive a novel
extremal principle which differs from the traditional one which dates back to \cite{KruMor80}. On the one hand, its prerequisites regarding
the position of the involved sets relative to each other is slightly more restrictive than in \cite{KruMor80} when the classical notion of
extremality, meaning that the sets of interest can be ``pushed apart from each other'', is used.
On the other hand, our new extremal principle covers settings where extremality is based on functions
which are just lower semicontinuous, and, thus,
applies in more general situations.
The final part of the paper is dedicated to the study of optimization problems with so-called geometric constraints,
where the feasible set equals the preimage of a closed set under a smooth transformation, whose objective function is
the sum of a smooth part and a merely lower semicontinuous part. First, we apply our concept of approximate stationarity
to this problem class in order to obtain necessary optimality conditions. Furthermore, we introduce an associated
qualification condition which guarantees
M-stationarity of approximately stationary points.
As we will show, this generalizes related considerations from \cite{ChenGuoLuYe2017,GuoYe2018} which were done in
a completely finite-dimensional setting.
Second, we suggest an augmented Lagrangian method for the numerical solution of geometrically constrained programs and
show that it computes approximately stationary points in our new sense. Finally, we use our theory in order to state
necessary optimality conditions for optimal control problems with a non-Lipschitzian so-called \emph{sparsity-promoting} term in the
objective function, see \cite{ItoKunisch2014,Wachsmuth2019}, which enforces optimal controls to be zero on large
parts of the domain.

The remaining parts of the paper are organized as follows.
In \cref{sec:notation}, we comment on the notation which is used in this manuscript and recall some
fundamentals from variational analysis.
\Cref{sec:semicontinuity} is dedicated to the study of a new notion of lower semicontinuity of
an extended-real-valued function relative to a given set-valued mapping or set.
We derive necessary optimality conditions of approximate stationarity type for rather general
optimization problems in \cref{sec:main_result}.
This is used in \cref{sec:generalized_separation} in order to derive a novel extremal principle
in generalized set separation.
Furthermore, we apply our findings from \cref{sec:main_result} in \cref{sec:geometric_constraints} in
order to state necessary optimality conditions of approximate stationarity type for optimization
problems in Banach spaces with geometric constraints and a composite objective function. Based on that, we derive a new qualification
condition ensuring M-stationarity of local minimizers, see \cref{sec:asymptotic_stuff}, an augmented Lagrangian method which
naturally computes approximately stationary points, see \cref{sec:alm}, and necessary optimality conditions for optimal
control problems with a sparsity-promoting term in the objective function, see \cref{sec:control}.
Some concluding remarks close the paper in \cref{sec:conclusions}.

\section{Notation and preliminaries}\label{sec:notation}

\subsection{Basic notation}
Our basic notation is standard, see e.g.\ \cite{Ioffe2017,Mordukhovich2006,RockafellarWets1998}.
The symbols $\R$ and $\N$ denote the sets of all real numbers and all positive integers, respectively.
Throughout the paper, $X$ and $Y$ are
either metric or Banach spaces
(although many facts, particularly, most of the definitions in \cref{sec:VA}, are valid in arbitrary normed vector spaces, i.e.,
do not require the spaces to be complete).
For brevity, we use the same notations $d(\cdot,\cdot)$ and $\|\cdot\|$ for distances and norms
in all spaces.
Banach spaces are often treated as metric spaces with the distance determined by the norm
in the usual way.
The distance from a point $x\in X$ to a set $\Omega\subset X$ in a metric space $X$ is defined by
$\dist_\Omega(x):=\inf_{u\in\Omega}d(x,u)$, and we use the convention $\dist_\varnothing(x) := +\infty$.
Throughout the paper, $\overline\Omega$ and $\intr\Omega$ denote the closure and the interior of $\Omega$, respectively.
Whenever $X$ is a Banach space, $\{x_k\}_{k\in\N}\subset X$ is a sequence, and $\bar x\in X$ is some
point, we exploit $x_k\to\bar x$ ($x_k\weakly\bar x$) in order to denote the strong (weak) convergence
of $\{x_k\}_{k\in\N}$ to $\bar x$. Similarly, we use $x_k^*\weaklystar x^*$ in order to express that
a sequence $\{x_k^*\}_{k\in\N}\subset X^*$ converges weakly$^*$ to $x^*\in X^*$.
Finally, $x_k\to_\Omega\bar x$ means that $\{x_k\}_{k\in\N}\subset\Omega$ converges
strongly to $\bar x$.
In case where $X$ is a Hilbert space and $K\subset X$ is a closed, convex set, we denote by
$P_K\colon X\to X$ the projection map associated with $K$.

If $X$ is a Banach space, its topological dual is denoted by $X^*$, while
$\langle\cdot,\cdot\rangle\colon X^*\times X\to\R$
denotes the bilinear form defining the pairing between the two spaces.
If not explicitly stated otherwise, products of (primal) metric or Banach spaces are equipped
with the maximum distances or norms, e.g., $\|(x,y)\|:=\max(\|x\|,\|y\|)$ for all $(x,y)\in X\times Y$.
Note that the corresponding dual norm is the sum norm given by $\|(x^*,y^*)\|:=\|x^*\|+\|y^*\|$ for all
$(x^*,y^*)\in X^*\times Y^*$.
The open unit balls in the primal and dual spaces are denoted by $\B$ and $\B^*$, respectively,
while the corresponding closed unit balls are denoted by $\overline{\B}$ and $\overline{\B}{}^*$,
respectively.
The notations $B_\delta(x)$ and $\overline{B}_\delta(x)$ stand, respectively,
for the open and closed balls with center $x$ and radius $\delta>0$ in $X$.

For an extended-real-valued function $\varphi\colon X\to\R_\infty:=\R\cup\{+\infty\}$,
its domain and epigraph are defined by
$\dom \varphi:=\{x\in X\,|\,\varphi(x)< +\infty\}$ and
$\epi\varphi:=\{(x,\mu)\in X\times\R\,|\,\varphi(x)\le\mu\}$, respectively.
For each set $\Omega\subset X$, we set $\varphi_\Omega:=\varphi+i_\Omega$
where $i_\Omega\colon X\to\R_\infty$
is the so-called indicator function of $\Omega$ which equals zero on $\Omega$
and is set to $+\infty$ on $X\setminus\Omega$.

A set-valued mapping $\Upsilon\colon X\rightrightarrows Y$ between metric spaces $X$ and $Y$ is a mapping,
which assigns to every $x\in X$ a (possibly empty) set $\Upsilon(x)\subset Y$.
We use the notations
$\gph \Upsilon:=\{(x,y)\in X\times Y\,|\,y\in \Upsilon(x)\}$,
$\Im\Upsilon:=\bigcup_{x\in X}\Upsilon(x)$,
and $\dom \Upsilon:=\{x\in X\,|\,\Upsilon(x)\ne\varnothing\}$
for the graph, the image, and the domain of $\Upsilon$, respectively.
Furthermore, $\Upsilon^{-1}\colon Y\rightrightarrows X$ given by
$\Upsilon^{-1}(y) :=\{x\in X\,|\,y\in \Upsilon(x)\}$ for all $y\in Y$ is referred to as
the inverse of $\Upsilon$.
Assuming that $\bar x\in\dom \Upsilon$ is fixed,
\[
	\limsup\limits_{x\to \bar x}\Upsilon(x)
	:=
	\left\{
		y\in Y\,\middle|\,
\exists\{(x_k,y_k)\}_{k\in\N}\subset\gph\Upsilon\colon\; x_k\to \bar x,\,y_k\to y
	\right\}
\]
is referred to as the (strong) outer limit of $\Upsilon$ at $\bar x$.
Finally, if $X$ is a Banach space, for a set-valued mapping $\Xi\colon X\tto X^*$ and $\bar x\in \dom\Xi$, we use
\[
	\wStarlimsup\limits_{x\to \bar x}\Xi(x)
	:=
	\left\{
		x^*\in X^*\,\middle|\,
\exists\{(x_k,x_k^*)\}_{k\in\N}\subset\gph\Xi\colon\; x_k\to \bar x,\,x_k^*\weaklystar x^*
	\right\}
\]
in order to denote the outer limit of $\Xi$ at $\bar x$ when equipping $X^*$ with
the weak$^*$ topology.
Let us note that both outer limits from above are limits in the sense of Painlev\'{e}--Kuratowski.

Recall that a Banach space is a so-called Asplund space if every continuous, convex function
on an open convex set is Fr\'echet differentiable on a dense subset, or equivalently,
if the dual of each separable subspace is separable as well.
We refer the reader to \cite{Phelps1993,Mordukhovich2006} for discussions about and characterizations
of Asplund spaces.
We would like to note that all reflexive, particularly, all finite-dimensional Banach spaces
possess the Asplund property.

\subsection{Variational analysis}
\label{sec:VA}

The subsequently introduced notions of variational analysis and generalized differentiation are
standard, see e.g.\ \cite{Kru03,Mordukhovich2006}.

Given a subset $\Omega$ of a Banach space $X$, a point $\bar x\in\Omega$, and a number $\eps\ge0$,
the nonempty, closed, convex set
\begin{equation}\label{eq:eps_normals}
	N_{\Omega,\eps}(\bar x)
	:=
	\left\{x^\ast\in X^\ast\,\middle|\,
	\limsup_{x\to_\Omega\bar x,\,x\neq\bar x}
		\frac {\langle x^\ast,x-\bar x\rangle}{\norm{x-\bar x}} \leq\eps
	\right\}
\end{equation}
is the set of $\eps$-normals to $\Omega$ at $\bar x$.
In case $\eps=0$, it is a closed, convex cone called
Fr\'{e}chet normal cone to $\Omega$ at $\bar x$.
In this case, we drop the subscript $\eps$ in the above notation and simply write
\begin{align*}
	N_{\Omega}(\bar x)
	:=
	\left\{x^\ast\in X^\ast\,\middle|\,
	\limsup_{x\to_\Omega\bar x,\,x\neq\bar x}
		\frac{\langle x^\ast,x-\bar x\rangle}{\norm{x-\bar x}} \leq 0
	\right\}.
\end{align*}
Based on \eqref{eq:eps_normals}, one can define the more robust
limiting normal cone to $\Omega$ at $\bar x$
by means of a limiting procedure:
\begin{align*}
	\overline{N}_{\Omega}(\bar x)
	:=
	\wStarlimsup\limits_{x\to_\Omega\bar x,\,\eps\downarrow 0}
	N_{\Omega,\eps}(x).
\end{align*}
Whenever $X$ is an Asplund space, the above definition admits the following simplification:
\begin{equation*}
	\overline{N}_{\Omega}(\bar x)= \wStarlimsup\limits_{x\to_\Omega\bar x} N_{\Omega}(x).
\end{equation*}
If $\Omega$ is a convex set, the Fr\'{e}chet and limiting normal cones reduce to the normal cone
in the sense of convex analysis, i.e.,
\begin{align*}
	N_{\Omega}(\bar x)
	=
	\overline N_{\Omega}(\bar x)
	=
	\left\{x^*\in X^*\,\middle|\,\langle x^*,x-\bar x \rangle \leq 0 \,\forall x\in \Omega\right\}.
\end{align*}

For a lower semicontinuous function $\varphi\colon X\to\R_{\infty}$, defined on a Banach space $X$,
its Fr\'echet subdifferential at $\bar x\in\dom \varphi$ is defined as
\begin{equation*}
	\begin{aligned}
	\partial \varphi(\bar x)
	:&=
	\left\{x^*\in X^*\,\middle|\, \liminf_{x\to\bar x,\,x\neq\bar x}
		\frac{\varphi(x)-\varphi(\bar x)-\langle x^*,x-\bar x\rangle}{\norm{x-\bar x}}\geq 0
	\right\}\\
	&=
	\left\{x^*\in X^*\,\middle|\, (x^*,-1)\in N_{\epi \varphi}(\bar x,\varphi(\bar x))
	\right\}.
	\end{aligned}
\end{equation*}
The limiting and singular limiting subdifferential of $\varphi$ at $\bar x$ are defined, respectively,
by means of
\begin{align*}
	\bsd \varphi(\bar x)
	&:=
	\left\{x^*\in X^*\,\middle|\, (x^*,-1)\in \overline{N}_{\epi \varphi}(\bar x,\varphi(\bar x))
	\right\},\\
	\bsd^\infty \varphi(\bar x)
	&:=
	\left\{x^*\in X^*\,\middle|\, (x^*,0)\in \overline{N}_{\epi \varphi}(\bar x,\varphi(\bar x))
	\right\}.
\end{align*}
Note that in case where $X$ is an Asplund space, we have
\begin{align*}
	\bsd \varphi(\bar x)
	&=
	\wStarlimsup\limits_{x\to\bar x,\,\varphi(x)\to\varphi(\bar x)}
	\partial\varphi(x),\\
	\bsd^\infty\varphi(\bar x)
	&=
	\wStarlimsup\limits_{x\to\bar x,\,\varphi(x)\to\varphi(\bar x),\,t\downarrow 0}
	t\,\partial\varphi(x),	
\end{align*}
see \cite[Theorems~2.34 and 2.38]{Mordukhovich2006}.
If $\varphi$ is convex, the Fr\'{e}chet and limiting subdifferential
reduce to the subdifferential in the sense of convex analysis, i.e.,
\begin{align*}
	\partial\varphi(\bar x)
	=
	\bsd\varphi(\bar x)
	=
	\left\{x^*\in X^*\,\middle|\,
		\varphi(x)-\varphi(\bar x)-\langle{x}^*,x-\bar x\rangle\ge 0\,\forall x\in X
	\right\}.
\end{align*}

By convention, we set
$N_{\Omega}(x)=\overline{N}_{\Omega}(x):=\varnothing$ if
$x\notin\Omega$
and
$\partial{\varphi}(x)=\bsd{\varphi}(x)=\bsd^\infty{\varphi}(x):=\varnothing$ if $x\notin\dom \varphi$.
It is easy to check that $N_{\Omega}(\bar x)=\partial i_\Omega(\bar x)$
and $\overline{N}_{\Omega}(\bar x)=\bsd i_\Omega(\bar x)$.

For a set-valued mapping $\Upsilon\colon X\rightrightarrows Y$ between Banach spaces,
its Fr\'{e}chet coderivative at $(\bar x,\bar y)\in\gph \Upsilon$ is defined as
\begin{align*}
	\forall y^*\in Y^*\colon\quad
	{D}^*\Upsilon(\bar x,\bar y)(y^*):=
	\left\{x^*\in X^*\,\middle|\, (x^*,-y^*)\in N_{\gph \Upsilon}(\bar x,\bar y)
	\right\}.
\end{align*}	

The proof of our main result \cref{thm:main_result} relies on certain fundamental results of variational analysis:
Ekeland's variational principle,
see e.g.\ \cite[Section~3.3]{AubinFrankowska2009} or \cite{Ekeland1974},
and two types of subdifferential sum rules
which address the subdifferential in the sense of convex analysis, see e.g.\
\cite[Theorem~3.16]{Phelps1993}, and the
Fr\'{e}chet subdifferential, see e.g.\ \cite[Theorem~3]{Fab89}.
Below, we provide these results for completeness.

\begin{lemma}\label{lem:Ekeland}
	Let $X$ be a complete metric space,
	$\varphi\colon X\to\R_{\infty}$ be lower semicontinuous and bounded from below,
	$\bx\in\dom \varphi$, and $\varepsilon>0$.
	Then there exists a point $\hat x\in X$ which satisfies the following conditions:
	\begin{enumerate}
		\item
			$\varphi(\hat x)\le \varphi(\bx)$;

		\item
			$\forall x\in X\colon\quad \varphi(x)+\varepsilon d(x,\hat x)\ge \varphi(\hat x)$.
	\end{enumerate}
\end{lemma}

\begin{lemma}\label{lem:SR}
	Let $X$ be a Banach space,
	$\varphi_1,\varphi_2\colon X\to\R_\infty$,
	and $\bar x\in\dom \varphi_1\cap\dom \varphi_2$.
	Then the following assertions hold.
	\begin{enumerate}
	\item\label{item:convex_sum_rule}
		\textbf{Convex sum rule}.
		Let $\varphi_1$ and $\varphi_2$ be convex, and $\varphi_1$ be continuous at a point in $\dom \varphi_2$.
		Then $\partial(\varphi_1+\varphi_2)(\bar x)=\partial \varphi_1(\bar x)+\partial \varphi_2(\bar x)$.
	\item\label{item:fuzzy_sum_rule}
		\textbf{Fuzzy sum rule}.
		Let $X$ be Asplund, $\varphi_1$ be Lipschitz continuous around $\bar x$,
		and $\varphi_2$ be lower semicontinuous in a neighborhood of $\bar x$.
		Then, for each $x^*\in\partial(\varphi_1+\varphi_2)(\bar x)$ and $\varepsilon>0$,
		there exist $x_1,x_2\in X$
		with $\norm{x_i-\bar x}<\varepsilon$ and $|\varphi_i(x_i)-\varphi_i(\bar x)|<\varepsilon$,
		$i=1,2$, such that
		$x^*\in\partial \varphi_1(x_1) +\partial \varphi_2(x_2)+\varepsilon\B^\ast$.
	\end{enumerate}
\end{lemma}

We will need representations of the subdifferentials of the distance function collected in the
next lemma. These results are taken from
\cite[Proposition 1.30]{Kru03}, \cite[Theorem 4.40]{Ioffe2017}, and
\cite[Section~3.5.2, Exercise~6]{Pen13}.

\begin{lemma}\label{lem:subdifferential_distance_function}
	Let $X$ be a Banach space, $\Omega\subset X$ be nonempty and closed, and
	$\bx\in X$. Then the following assertions hold.
	\begin{enumerate}
	\item\label{item:sdf_distance_function_in_set_points}
		If $\bar x\in\Omega$, then $\partial\dist_\Omega(\bar x)=N_\Omega(\bar x)\cap\overline{\B}{}^*$.
	\item\label{item:sdf_distance_function_out_of_set_points}
		If $\bar x\notin\Omega$ and either $X$ is Asplund or $\Omega$ is convex, then,
		for each $x^*\in\partial\dist_\Omega(\bar x)$ and each $\eps>0$, there exist
		$x\in\Omega$ and $u^*\in N_\Omega(x)$
		such that $\norm{x-\bar x}<\dist_\Omega(\bar x)+\varepsilon$
		and $\norm{x^*-u^*}<\varepsilon$.
	\end{enumerate}
\end{lemma}

Let us briefly mention that assertion~\ref{item:sdf_distance_function_out_of_set_points} of \cref{lem:subdifferential_distance_function}
can obviously be improved when the set of projections of $\bar x$ onto $\Omega$ is nonempty,
see \cite[Proposition~1.102]{Mordukhovich2006}.
This is always the case if $\Omega$ is a nonempty, closed, convex subset of a reflexive Banach space,
since in this case $\Omega$ is weakly sequentially compact while the norm is weakly sequentially lower
semicontinuous.

The conditions in the final definition of this subsection are standard,
see e.g.\ \cite{KlatteKummer2002,Kru09}.

\begin{definition}\label{def:stationarity}
Let $X$ be a metric space, $\varphi\colon X\to\R_\infty$, and $\bx\in\dom \varphi$.
\begin{enumerate}
	\item
		We call $\bx$ a \emph{stationary point of $\varphi$} if
		$\liminf_{ x\to\bx,\,x\neq\bar x}\frac{\varphi(x)-\varphi(\bx)}{d(x,\bx)}\geq 0$.
	\item
		Let $\eps>0$ and $U\subset X$ with $\bar x\in U$.
		We call $\bx$ an \emph{$\eps$-minimal point of $\varphi$ on $U$} if
		$\inf_{x\in U}\varphi(x)>\varphi(\bx)-\eps$.
		If $U=X$, $\bar x$ is called a \emph{globally $\varepsilon$-minimal point of $\varphi$}.
\end{enumerate}
\end{definition}

In the subsequent remark, we interrelate the concepts from \cref{def:stationarity}.
\begin{remark}\label{rem:minimality_vs_stationarity}
	For a metric space $X$, $\varphi\colon X\to\R_\infty$, and $\bx\in\dom \varphi$, the following assertions hold.
	\begin{enumerate}
		\item
			If $\bar x$ is a local minimizer of $\varphi$, then it is a stationary point of $\varphi$.
		\item
			If $\bar x$ is a stationary point of $\varphi$, then, for each $\varepsilon>0$
			and each sufficiently small $\delta>0$,
			$\bar x$ is an $\varepsilon\delta$-minimal point of $\varphi$
			on $B_\delta(\bar x)$.
		\item
			If $X$ is a normed
			space, then $\bar x$ is a stationary point of $\varphi$ if and only if $0\in\sd\varphi(\bar x)$.
	\end{enumerate}
\end{remark}

\section{Novel notions of semicontinuity}\label{sec:semicontinuity}

In this paper, we exploit new notions of lower semicontinuity of extended-real-valued functions
relative to a given set-valued mapping or set.
Here, we first introduce the concepts of interest before
studying their properties and presenting sufficient conditions for their validity.

\subsection{Lower semicontinuity of a function relative to a set-valued mapping or set}

Let us start with the definition of the property of our interest.

\begin{definition}\label{def:lower_semicontinuity_relative_to_svm}
Fix metric spaces $X$ and $Y$,
$\Phi\colon X\rightrightarrows Y$,
$\varphi\colon X\to\R_\infty$,
and $\by\in Y$.
\begin{enumerate}
\item\label{item:def_lower_semicontinuity_on_a_set}
Let a subset $U\subset X$ be such that $U\cap\Phi^{-1}(\by)\cap\dom\varphi\ne\varnothing$.
The function $\varphi$ is
\emph{lower semicontinuous on $U$ relative to $\Phi$ at $\by$} if
\begin{equation}
\label{eq:estimate_lsc_wrt_set_valued_map}
	\inf_{u\in \Phi^{-1}(\by)\cap U}\varphi(u)
	\le
	\inf_{\substack{U'+\rho\B\subset U,\\\rho>0}}
	\liminf_{\substack{x\in U',\,y\to\by,\\
	\dist_{\gph \Phi}(x,y)\to0}}\varphi(x).
\end{equation}
\item\label{item:def_lower_semicontinuity_near_point}
Let $\bx\in\Phi^{-1}(\by)\cap\dom\varphi$.
The function $\varphi$ is
\emph{lower semicontinuous near $\bx$ relative to $\Phi$ at $\by$} if there is a $\de_0>0$ such that, for each $\de\in(0,\de_0)$,
$\varphi$ is lower semicontinuous on $\overline{B}_\de(\bx)$ relative to $\Phi$ at $\by$.
\end{enumerate}
\end{definition}

Inequality \eqref{eq:estimate_lsc_wrt_set_valued_map} can be strict, see \cref{ex:strict_estimate_lsc} below.
Note that whenever \eqref{eq:estimate_lsc_wrt_set_valued_map} holds with a subset $U\subset X$, it also holds with $\overline{U}$ in place of $U$.
The converse implication is not true in general, see \cref{ex:lower_semicontinuity_on_smaller_set} below.
Particularly, a function which is lower semicontinuous on a set $U$ relative to $\Phi$ at $\bar y$ may fail to have this property on a smaller set.
This shortcoming explains the idea behind \cref{def:lower_semicontinuity_relative_to_svm}\,\ref{item:def_lower_semicontinuity_near_point}.
Furthermore, we have the following result.
\begin{lemma}\label{lem:lower_semicontinuity_for_smaller_radii}
	Fix metric spaces $X$ and $Y$,
	$\Phi\colon X\rightrightarrows Y$,
	$\varphi\colon X\to\R_\infty$,
	$(\bx,\by)\in\gph\Phi$, and a subset $U\subset X$ with $\bx\in U\cap\dom\varphi$.
	Assume that $\bx$ is a minimizer of $\varphi$ on $U$.
	If $\varphi$ is lower semicontinuous on $U$ relative to $\Phi$ at $\bar y$, then it is
	lower semicontinuous on $\widehat U$ relative to $\Phi$ at $\bar y$ for each subset $\widehat U$ satisfying $\bx\in\widehat U\subset U$.
\end{lemma}

\begin{proof}
	For each subset $\widehat U$ satisfying $\bx\in\widehat U\subset U$, we find
	\begin{align*}
		\inf\limits_{u\in\Phi^{-1}(\by)\cap\widehat U}\varphi(u)
		&=
		\varphi(\bar x)
		=
		\inf\limits_{u\in \Phi^{-1}(\by)\cap U}\varphi(u)
		\\
		&
	\le
	\inf_{\substack{U'+\rho\B\subset U,\\\rho>0}}
	\liminf_{\substack{x\in U',\,y\to\by,\\
	\dist_{\gph \Phi}(x,y)\to0}}\varphi(x)
	\le
	\inf_{\substack{U'+\rho\B\subset\widehat U,\\\rho>0}}
	\liminf_{\substack{x\in U',\,y\to\by,\\
	\dist_{\gph \Phi}(x,y)\to0}}\varphi(x),
	\end{align*}
	which shows the claim.
\end{proof}

The properties in the next definition
are particular cases of the ones in \cref{def:lower_semicontinuity_relative_to_svm},
corresponding to the set-valued mapping\ $\Phi\colon X\rightrightarrows Y$ whose graph is given by
$\gph \Phi:=\Omega\times Y$, where
$\Omega\subset X$ is a fixed set and $Y$ can be an arbitrary metric space, e.g., one can take $Y:=\R$.
Observe that in this case, $\Phi^{-1}(y)=\Omega$ is valid for all $y\in Y$.

\begin{definition}\label{def:lower_semicontinuity_relative_to_set}
Fix a metric space $X$,
$\varphi\colon X\to\R_\infty$,
and $\Omega\subset X$.
\begin{enumerate}
\item\label{item:def_lower_semicontinuity_relative_to_set}
Let a subset $U\subset X$ be such that $U\cap\Omega\cap\dom\varphi\ne\varnothing$.
The function $\varphi$ is
\emph{lower semicontinuous on $U$ relative to $\Omega$} if
\begin{equation}
\label{eq:lower_semicontinuity_relative_to_set-1}
	\inf_{u\in\Omega\cap U}\varphi(u)
	\le
	\inf_{\substack{U'+\rho\B\subset U,\\\rho>0}}
	\liminf_{\substack{x\in U',\\
	\dist_{\Omega}(x)\to0}}\varphi(x).
\end{equation}
\item\label{item:def_lower_semicontinuity_relative_to_set_near_point}
Let $\bx\in\Omega\cap\dom\varphi$.
The function $\varphi$ is
\emph{lower semicontinuous near $\bx$ relative to $\Omega$} if there is a $\de_0>0$ such that, for each $\de\in(0,\de_0)$,
$\varphi$ is lower semicontinuous on $\overline{B}_\de(\bx)$ relative to $\Omega$.
\end{enumerate}
\end{definition}

The subsequent example shows that \eqref{eq:lower_semicontinuity_relative_to_set-1} can be strict.
\begin{example}
\label{ex:strict_estimate_lsc}
Consider the lower semicontinuous function $\varphi\colon\R\to\R$ given by
$\varphi(x):=0$ if $x\leq 0$ and
$\varphi(x):=1$ if $x> 0$, and the sets $\Omega=U:=[0,1]\subset\R$.
Then $\inf_{u\in\Omega\cap U}\varphi(u)=0$, while if a subset $U'$ satisfies $U'+\rho\B\subset U$ for some $\rho>0$,
then $U'\subset(0,1)$, and consequently $\varphi(x)=1$ for all $x\in U'$.
Hence, the right-hand side\ of \eqref{eq:lower_semicontinuity_relative_to_set-1} equals $1$.
\end{example}

A function which is lower semicontinuous on a set $U$ relative to $\Omega$ may fail to have this property on a smaller set.

\begin{example}
\label{ex:lower_semicontinuity_on_smaller_set}
Consider the function $\varphi\colon\R\to\R$ given by
$\varphi(x):=0$ if $x\leq 0$, and
$\varphi(x):=-1$ if $x> 0$,
the set $\Omega:=\{0,1\}\subset\R$, and the point $\bar x:=0$.
Consider the closed interval $U_1:=[-1,1]$.
We find $\inf_{u\in\Omega\cap U_1}\varphi(u)=-1$ which is the global minimal value of $\varphi$ on $\R$.
Hence, $\varphi$ is lower semicontinuous on $U_1$ relative to $\Omega$ by \cref{def:lower_semicontinuity_relative_to_set}.
For $U_2:=(-1,1)$, we find $\inf_{u\in \Omega\cap U_2}\varphi(u)=0$.
Moreover, choosing $U':=(-1/2,1/2)$ and $x_k:=1/(k+2)$ for each $k\in\N$, we find $U'+\tfrac12\mathbb B\subset U_2$,
$\{x_k\}_{k\in\N}\subset U'$, $d(x_k,\bar x)\to 0$, and $\varphi(x_k)\to-1$, i.e., $\varphi$ is not lower semicontinuous
on $U_2$ relative to $\Omega$ by definition.
Note that $\bar x$ is a local minimizer of $\varphi$ on $\Omega$ but not on $U_1$ or $U_2$.
\end{example}

In the next two statements, we present
sequential characterizations of the properties from \cref{def:lower_semicontinuity_relative_to_svm}\,\ref{item:def_lower_semicontinuity_on_a_set}
and \cref{def:lower_semicontinuity_relative_to_set}\,\ref{item:def_lower_semicontinuity_relative_to_set}.

\begin{proposition}
\label{prop:sequential_characterization_lsc}
Fix metric spaces $X$ and $Y$,
$\Phi\colon X\rightrightarrows Y$,
$\varphi\colon X\to\R_\infty$,
$\by\in Y$, and a subset $U\subset X$ with $U\cap\Phi^{-1}(\by)\cap\dom\varphi\ne\varnothing$.
Then $\varphi$ is lower semicontinuous on $U$ relative to $\Phi$ at $\by$ if
and only if
\begin{align*}
	\inf_{u\in \Phi^{-1}(\by)\cap U}\varphi(u)
	\le
	\liminf_{k\to+\infty}\varphi(x_k)
\end{align*}
for all sequences $\{(x_k,y_k)\}_{k\in\N}\subset X\times Y$ satisfying
$y_k\to\by$,
$\dist_{\gph \Phi}(x_k,y_k)\to0$,
and
$\{x_k\}_{k\in\N}
+\rho\B\subset U$ for some $\rho>0$.
\end{proposition}

\begin{proof}
We need to show that the right-hand side\ of \eqref{eq:estimate_lsc_wrt_set_valued_map} equals the infimum over all numbers $\liminf_{k\to+\infty}\varphi(x_k)$
where the sequence $\{(x_k,y_k)\}_{k\in\N}\subset X\times Y$ needs to satisfy
$y_k\to\by$,
$\dist_{\gph \Phi}(x_k,y_k)\to0$,
and $\{x_k\}_{k\in\N}+\rho\B\subset U$ for some $\rho>0$.
Let $\{(x_k,y_k)\}_{k\in\N}$ be such a sequence.
Then
\begin{align*}
	\inf_{\substack{U'+\rho\B\subset U,\\\rho>0}}
	\liminf_{\substack{x\in U',\,y\to\by,\\
	\dist_{\gph \Phi}(x,y)\to0}}
	\varphi(x)
	\le
	\liminf_{\substack{x\in \{x_k\}_{k\in\N},\,y\to\by,\\\dist_{\gph \Phi}(x,y)\to0}}
	\varphi(x)
	\le
	\liminf_{k\to+\infty}\varphi(x_k).
\end{align*}

Conversely, let the right-hand side\ of \eqref{eq:estimate_lsc_wrt_set_valued_map} be finite, and choose $\eps>0$ arbitrarily.
Then there exist a subset $\widehat U\subset U$ and a number $\hat\rho>0$ such that $\widehat U+\hat\rho\B\subset U$ and
\begin{align*}
	\liminf_{k\to+\infty}
	\inf_{\substack{x\in\widehat U,\,d(y,\by)<\frac1k,\\\dist_{\gph \Phi}(x,y)<\frac1k}}\varphi(x)
	=
	\liminf_{\substack{x\in\widehat U,\,y\to\by,\\ \dist_{\gph \Phi}(x,y)\to0}}
	\varphi(x)
	<
	\inf_{\substack{U'+\rho\B\subset U,\\\rho>0}}
	\liminf_{\substack{x\in U',\,y\to\by,\\
	\dist_{\gph \Phi}(x,y)\to0}}
	\varphi(x)
	+
	\eps.
\end{align*}
For each $k\in\N$ such that $\inf_{{x\in\widehat U,\,d(y,\by)<\frac1k,\,\dist_{\gph \Phi}(x,y)<\frac1k}}\varphi(x)$ is finite,
there is a tuple $(x_k,y_k)\in X\times Y$ such that $x_k\in\widehat U$, $d(y_k,\by)<1/k$,
$\dist_{\gph \Phi}(x_k,y_k)<1/k$, and
\begin{align*}
	\varphi(x_k)
	<
	\inf_{\substack{x\in\widehat U,\,d(y,\by)<\frac1k,\\\dist_{\gph \Phi}(x,y)<\frac1k}}
		\varphi(x)+\frac1k.
\end{align*}
Considering the tail of the sequences, if necessary, we have $\{x_k\}_{k\in\N}+\hat\rho\B\subset U$,
$y_k\to\by$,
$\dist_{\gph \Phi}(x_k,y_k)\to0$, and
\begin{align*}
	\liminf_{k\to+\infty}\varphi(x_k)
	<
	\inf_{\substack{U'+\rho\B\subset U,\\\rho>0}}
	\liminf_{\substack{x\in U',\,y\to\by,\\
	\dist_{\gph \Phi}(x,y)\to0}}
	\varphi(x)
	+
	\eps.
\end{align*}
As the number $\eps$ has been chosen arbitrarily, this proves the converse part in the present setting.
If the  right-hand side\ of \eqref{eq:estimate_lsc_wrt_set_valued_map} equals $-\infty$, then for each $M>0$, we find a subset $\widehat U\subset U$
and a number $\hat\rho>0$ such that $\widehat U+\hat\rho\B\subset U$ and
\begin{align*}
	\liminf\limits_{\substack{x\in\widehat U,\,y\to\by,\\\dist_{\gph \Phi}(x,y)\to 0}}\varphi(x)
	<
	-M.
\end{align*}
Hence, there is a sequence $\{(x_k,y_k)\}_{k\in\N}\subset X\times Y$ such that
$\{x_k\}_{k\in\N}+\hat\rho\B\subset U$, $y_k\to\bar y$, and $\dist_{\gph \Phi}(x_k,y_k)\to 0$
as $k\to+\infty$ while $\liminf_{k\to+\infty}\varphi(x_k)<-M$. Taking the infimum over all
$M>0$ now completes the proof of the assertion.
\end{proof}

\begin{corollary}
\label{cor:sequential_characterization_lsc_relative_to_set}
Let $X$ be a metric space, $\varphi\colon X\to\R_\infty$, and $\Omega,U\subset X$ be sets with $\Omega\cap U\cap\dom\varphi\ne\varnothing$.
Then $\varphi$ is lower semicontinuous on $U$ relative to $\Omega$ if and only if
\begin{equation}
\label{eq:sequential_characterization_lsc}
	\inf_{u\in\Omega\cap U}\varphi(u)
	\le
	\liminf_{k\to+\infty}\varphi(x_k)
\end{equation}
for all sequences $\{x_k\}_{k\in\N}\subset X$ satisfying
$\dist_{\Omega}(x_k)\to0$,
and $\{x_k\}_{k\in\N}+\rho\B\subset U$ for some $\rho>0$.
\end{corollary}

\subsection{Sufficient conditions for lower semicontinuity of a function relative to a set-valued mapping}

As we will demonstrate below,
the property from \cref{def:lower_semicontinuity_relative_to_svm}\,\ref{item:def_lower_semicontinuity_on_a_set} is valid whenever the involved function $\varphi$
and the set-valued mapping $\Phi$ enjoy certain
semicontinuity properties, i.e., it can be decomposed into two independent properties regarding the
two main
data objects.
This will be beneficial in order to identify scenarios where
the new concept applies.

The upper semicontinuity properties of a set-valued mapping that we
state in the following two definitions seem to fit well for this purpose
(in combination with the corresponding lower semicontinuity properties of a function).

\begin{definition}
\label{def:upper_semicontinuity}
Fix metric spaces $X$ and $Y$,
$S\colon Y\rightrightarrows X$, and $\by\in \dom S$.
The mapping $S$ is
\emph{upper semicontinuous at $\by$}
if
\begin{align*}
	\lim_{x\in S(y),\,y\to\by}\dist_{S(\by)}(x)=0.
\end{align*}
\end{definition}

\begin{definition}
\label{def:partial_weak_upper_semicontinuity}
Fix a Banach space $X$, a metric space $Y$, $S\colon Y\rightrightarrows X$,
and $\by\in \dom S$.
The mapping $S$ is
\emph{partially weakly sequentially upper semicontinuous at $\by$}
if $x\in S(\by)$
holds for each sequence $\{(y_k,x_k)\}_{k\in\N}\subset\gph S$ which satisfies $y_k\to\by$ and $x_k\weakly x$.
\end{definition}

For a discussion of the property in
\Cref{def:upper_semicontinuity}, we refer the reader to
\cite[p.~10]{KlatteKummer2002}.
The property in \cref{def:partial_weak_upper_semicontinuity} can be interpreted as the usual sequential upper semicontinuity
if $X$ is equipped with the weak topology.
In case where $Y$ is a Banach space, this property is inherent whenever the graph of the
underlying set-valued mapping is weakly sequentially closed which is naturally given whenever the latter is convex and closed.
Obviously, each closed-graph set-valued mapping with a finite-dimensional image space is partially weakly sequentially
upper semicontinuous.

\begin{proposition}
\label{prop:decomposition_lsc}
Fix metric spaces $X$ and $Y$, $\Phi\colon X\rightrightarrows Y$, and $\varphi\colon X\to\R_\infty$.
Let $\by\in Y$ and a subset $U\subset X$ with $U\cap\Phi^{-1}(\by)\cap\dom\varphi\ne\varnothing$
be arbitrarily chosen.
Define $S\colon Y\rightrightarrows X$ by $S(y):=\Phi^{-1}(y)\cap U$ for all $y\in Y$.
If one of the following criteria holds, then 
$\varphi$ is lower semicontinuous on $U$ relative to $\Phi$ at $\by$:
\begin{enumerate}
\item\label{item:lower_semicontinuity_via_decomposition}
$\varphi$ is lower semicontinuous on $U$ relative to $\Phi^{-1}(\by)$ and
$S$ is upper semicontinuous at $\by$;
\item\label{item:lower_semicontinuity_via_weak_sequential_properties}
$X$ is a reflexive Banach space,
$U$ is closed and convex,
$\varphi$ is weakly sequentially lower semicontinuous on $U$, and $S$ is
partially weakly sequentially upper semicontinuous at~$\by$.
\end{enumerate}
\end{proposition}

\begin{proof}
Let a sequence $\{(x_k,y_k)\}_{k\in\N}\subset X\times Y$ satisfying
$y_k\to\by$,
$\dist_{\gph \Phi}(x_k,y_k)\to0$,
and $\{x_k\}_{k\in\N}+\rho\B\subset U$ for some $\rho>0$ be arbitrarily chosen.
There exists a sequence $\{(x'_k,y'_k)\}_{k\in\N}\subset\gph \Phi$ such that $d((x'_k,y'_k),(x_k,y_k))\to0$.
Hence, $y'_k\to\by$ and, for all sufficiently large $k\in\N$, we have $d(x'_k,x_k)<\rho$, and, consequently, $x'_k\in U$.
\begin{enumerate}
\item
By \cref{def:upper_semicontinuity},
$\dist_{\Phi^{-1}(\by)}(x'_k)\to0$.
Then
$\dist_{\Phi^{-1}(\by)}(x_k)\to0$ and, by \cref{cor:sequential_characterization_lsc_relative_to_set},
inequality \eqref{eq:sequential_characterization_lsc} holds, where $\Omega:=\Phi^{-1}(\by)$.

\item
Passing to a subsequence (without relabeling), we can assume $x_k\weakly\hat x$ for some $\hat x\in\clconv\{x_k\}_{k\in\N}\subset U$
since $\{x_k\}_{k\in\N}$ is a bounded sequence of a reflexive Banach space and $U$ is convex as well as closed.
Hence, we find $\varphi(\hat x)\le\liminf_{k\to+\infty}\varphi(x_k)$
by weak sequential lower semicontinuity of $\varphi$.
Obviously, we have $x'_k\weakly\hat x$.
By \cref{def:partial_weak_upper_semicontinuity},
$\hat x\in \Phi^{-1}(\by)$ holds true.
Thus,
$\inf_{u\in \Phi^{-1}(\by)\cap U}\varphi(u)\le
\varphi(\hat x)\le\liminf_{k\to+\infty}\varphi(x_k)$.
\end{enumerate}
As the sequence
$\{(x_k,y_k)\}_{k\in\N}$ has been chosen arbitrarily, the conclusion follows from \cref{prop:sequential_characterization_lsc}.
\end{proof}

The next assertion is an immediate consequence of \cref{prop:decomposition_lsc}
with the conditions from~\ref{item:lower_semicontinuity_via_weak_sequential_properties}.

\begin{corollary}\label{cor:lower_semicontinuity_svm_via_weak_sequential_lower_semicontinuity}
Fix a reflexive Banach space $X$, a closed and convex set $U\subset X$, 
$\varphi\colon X\to\R_\infty$ which is weakly sequentially lower semicontinuous on $U$,
$\Phi\colon X\tto Y$ where $Y$ is another Banach space, and
some $\by\in Y$ such that $U\cap\Phi^{-1}(\by)\cap\dom\varphi\neq\varnothing$.
Then $\varphi$ is lower semicontinuous on $U$ relative to $\Phi$ at $\by$ provided that one
of the following conditions is satisfied:
\begin{enumerate}
	\item
		$\gph\Phi\cap(U\times Y)$ is weakly sequentially closed;
	\item
		$X$ is finite-dimensional and $\gph\Phi\cap(U\times Y)$ is closed.
\end{enumerate}
Particularly, whenever $\bar x\in\Phi^{-1}(\by)\cap\dom\varphi$ is fixed,
$\varphi$ is weakly sequentially lower semicontinuous, and either $\gph\Phi$
is weakly sequentially closed or $\gph\Phi$ is closed while $X$ is finite-dimensional,
then $\varphi$ is lower semicontinuous near $\bx$ relative to $\Phi$ at $\by$.
\end{corollary}

In the upcoming subsections, we discuss sufficient conditions for the semicontinuity properties of a set-valued mapping and
an extended-real-valued function appearing in the conditions~\ref{item:lower_semicontinuity_via_decomposition} of \cref{prop:decomposition_lsc}.

\subsection{Sufficient conditions for lower semicontinuity of a function relative to a set}

In the statement below, we present some simple situations where a function is lower semicontinuous
relative to a set in the sense of \cref{def:lower_semicontinuity_relative_to_set}\,\ref{item:def_lower_semicontinuity_relative_to_set}.

\begin{proposition}\label{prop:restricted_lower_semicontinuity}
Let $X$ be a metric space, $\varphi\colon X\to\R_\infty$, and $\Omega,U\subset X$ be sets with $\Omega\cap U\cap\dom\varphi\ne\varnothing$.
Then $\varphi$ is lower semicontinuous on $U$ relative to $\Omega$ provided that one of the following
conditions is satisfied:
\begin{enumerate}
\item
$U\subset\Omega$;
\item
$\Omega\cap U=\{\bar x\}$, and
$\varphi$ is lower semicontinuous at $\bar x$;
\item
$\bar x\in\Omega\cap U$ is a
minimizer of $\varphi$ on $U$;
\item
$\varphi$ is uniformly continuous on $U$.
\end{enumerate}
\end{proposition}

\begin{proof}
Under each of the condition (a), (b), and (c), the conclusion is straightforward since inequality \eqref{eq:lower_semicontinuity_relative_to_set-1}
is an immediate consequence of the following simple relations, respectively, holding with any $U'\subset U$:
\begin{enumerate}
\item
$\inf\limits_{u\in\Omega\cap U}\varphi(u)=\inf\limits_{u\in U}\varphi(u)$,
$\liminf\limits_{\substack{x\in U',\,
\dist_{\Omega}(x)\to0}}\varphi(x)=\inf\limits_{x\in U'}\varphi(x) \ge\inf\limits_{x\in U}\varphi(x)$;
\item
$\inf\limits_{u\in\Omega\cap U}\varphi(u)=\varphi(\bx)$,
$\liminf\limits_{\substack{x\in U',\,
\dist_{\Omega}(x)\to0}}\varphi(x)
=\liminf\limits_{x\to_{U'}\bx}\varphi(x)
\ge\liminf\limits_{x\to\bx}\varphi(x)\ge\varphi(\bx)$;
\item
$\inf\limits_{u\in\Omega\cap U}\varphi(u)=\varphi(\bx)$,
$\liminf\limits_{\substack{x\in U',\,
\dist_{\Omega}(x)\to0}}\varphi(x) \ge\varphi(\bx)$.
\end{enumerate}
It remains to prove the claim under
condition (d).
Let a number $\varepsilon>0$ be arbitrarily chosen.
Let a subset $U'\subset X$ and a number $\rho>0$ be such that $U'+\rho\B\subset U$.
By (d), there is a $\de>0$ such that
\[
\forall x,x'\in U\colon\quad
d(x,x')<\de\quad\Longrightarrow\quad |\varphi(x)-\varphi(x')|<\varepsilon.
\]
Let a point $x\in U'$ satisfy $\dist_\Omega(x)<\de':=\min(\rho,\de)$.
Then there is a point $x'\in\Omega$ satisfying $d(x,x')<\de'$.
Hence, $x,x'\in U$, $d(x,x')<\de$, and, consequently, $|\varphi(x)-\varphi(x')|<\varepsilon$.
Thus, we have
$\inf_{u\in\Omega\cap U}\varphi(u)\leq\varphi(x')
<\varphi(x)+\varepsilon$, and, consequently,
\begin{align*}
	\inf_{u\in\Omega\cap U}\varphi(u)
	\le
	\liminf_{{x\in U',\,
	\dist_{\Omega}(x)\to0}}\varphi(x)+\varepsilon.
\end{align*}
Taking the infimum on the right-hand side\ of the last inequality over $\eps$ and $U'$, we arrive at \eqref{eq:lower_semicontinuity_relative_to_set-1}.
\end{proof}

As a corollary, we obtain sufficient conditions for the lower semicontinuity property from
\cref{def:lower_semicontinuity_relative_to_set}\,\ref{item:def_lower_semicontinuity_relative_to_set_near_point}.
\begin{corollary}
\label{cor:sufficient_criteria_lower_semicontinuity_relative_to_a_set}
Let $X$ be a metric space, $\varphi\colon X\to\R_\infty$, $\Omega\subset X$, and $\bx\in\Omega\cap\dom\varphi$.
Then $\varphi$ is lower semicontinuous near $\bx$ relative to $\Omega$ provided that one of the following
conditions is satisfied:
\begin{enumerate}
\item\label{item:lsc_via_interiority}
$\bar x\in\intr\Omega$;
\item\label{item:lsc_via_isolatedness}
$\bar x$ is an isolated point of $\Omega$, and
$\varphi$ is lower semicontinuous at $\bar x$;
\item\label{item:lsc_via_local_minimality}
$\bar x$ is an (unconditional) local minimizer of $\varphi$;
\item\label{item:lsc_via_uniform_continuity}
$\varphi$ is uniformly continuous near $\bx$.
\end{enumerate}
\end{corollary}

It follows from \cref{cor:sufficient_criteria_lower_semicontinuity_relative_to_a_set}\,\ref{item:lsc_via_uniform_continuity}
that each locally Lipschitz function is lower semicontinuous near a reference point relative
to any set containing this point.

The subsequent result can be directly distilled from
\cref{cor:lower_semicontinuity_svm_via_weak_sequential_lower_semicontinuity}.

\begin{proposition}\label{prop:lower_semicontinuity_set_via_weak_sequential_lower_semicontinuity}
	Fix a reflexive Banach space $X$, a closed and convex set $U\subset X$, and $\varphi\colon X\to\R_\infty$
	which is weakly sequentially lower semicontinuous on $U$.
	Let $\Omega\subset X$ be chosen such that $\Omega\cap U\cap\dom\varphi\neq\varnothing$ while $\Omega\cap U$ is weakly
	sequentially closed. Then $\varphi$ is lower semicontionuous on $U$ relative to $\Omega$.
\end{proposition}

As a corollary, we obtain the subsequent result.
\begin{corollary}\label{cor:lower_semicontinuity_near_point_set_via_weak_sequential_lower_semicontinuity}
	Fix a reflexive Banach space $X$, $\varphi\colon X\to\R_\infty$
	which is weakly sequentially lower semicontinuous, and a weakly sequentially closed set $\Omega\subset X$.
	Then, for each $\bar x\in\Omega\cap\dom\varphi$, $\varphi$ is lower semicontinuous near $\bar x$ relative to $\Omega$.
\end{corollary}

Note that whenever $X$ is finite-dimensional, $\varphi\colon X\to\R_\infty$ is lower semicontinuous, and $\Omega\subset X$
is closed, then the assumptions of \cref{cor:lower_semicontinuity_near_point_set_via_weak_sequential_lower_semicontinuity}
hold trivially.

The following statement shows that lower semicontinuity relative to a set is preserved under decoupled
summation.

\begin{proposition}\label{lem:decoupled_sum_rule}
Fix $n\in\N$ with $n\geq 2$.
For each $i\in\{1,\ldots,n\}$, let $X_i$ be a metric space, $\varphi_i\colon X_i\to\R_\infty$, $\Omega_i,U_i\subset X_i$, 
and $\Omega_i\cap U_i\cap\dom\varphi_i\ne\varnothing$.
Suppose that $\varphi_i$ is lower semicontinuous on $U_i$
relative to $\Omega_i$.
Then $\varphi\colon X_1\times\ldots\times X_n\to\R_\infty$ given by
\[
\forall (x_1,\ldots,x_n)\in X_1\times\ldots\times X_n\colon\quad
\varphi(x_1,\ldots,x_n):=\varphi_1(x_1)+\ldots+\varphi_n(x_n)
\]
is lower semicontinuous on $U:=U_1\times\ldots\times U_n$
relative to $\Omega:=\Omega_1\times\ldots\times\Omega_n$.
\end{proposition}

\begin{proof}
The assertion is a direct consequence of \cref{def:lower_semicontinuity_relative_to_set}\,\ref{item:def_lower_semicontinuity_relative_to_set}.
More precisely, we find
\begin{align*}
	\inf_{u\in\Omega\cap U}\varphi(u)
=\sum_{i=1}^n\inf_{u_i\in\Omega_i\cap U_i}\varphi_i(u_i)
&\le\sum_{i=1}^n
	\inf_{\substack{U_i'+\rho_i\B\subset U_i,\\\rho_i>0}}
	\liminf_{\substack{x_i\in U_i',\\
	\dist_{\Omega_i}(x_i)\to0}}\varphi_i(x_i)
\\&
	=	\inf_{\substack{U'+\rho\B\subset U,\\\rho>0}}
	\liminf_{\substack{x\in U',\\
	\dist_{\Omega}(x)\to0}}\varphi(x),
\end{align*}
and this proves the claim.
\end{proof}

\subsection{A sufficient condition for upper semicontinuity of the inverse of a set-valued mapping}\label{sec:sufficient_condition_upper_semicontinuity}

The next statement presents a condition ensuring validity of the upper semicontinuity assumption
which appears in \cref{prop:decomposition_lsc}\,\ref{item:lower_semicontinuity_via_decomposition}.

\begin{proposition}\label{prop:sufficient_condition_for_quasi_upper_semicontinuity_feasiblity_map}
	Let $X$ and $Y$ be metric spaces, $\Phi\colon X\tto Y$, and $(\bx,\by)\in\gph\Phi$.
	Assume that $\Phi$ is metrically subregular at $(\bx,\by)$, i.e., that there exist a neighborhood $U$ of $\bx$ and a constant
	$L>0$ such that
	\begin{equation}\label{eq:metric_subregularity}
		\forall x\in U\colon\quad
		\dist_{\Phi^{-1}(\by)}(x)\leq L\,\dist_{\Phi(x)}(\by).
	\end{equation}
	Then, for each set $U'\subset U$ satisfying $\bx\in U'$, the mapping $S_{U'}\colon Y\tto X$, given by
	$S_{U'}(y):=\Phi^{-1}(y)\cap U'$ for each $y\in Y$, is upper semicontinuous at $\by$.
\end{proposition}
\begin{proof}
	Let a number $\varepsilon>0$ as well as $U'\subset U$ with $\bar x\in U'$ be given.
	Choose a number $\de\in(0,\eps/L)$.
	Then, for each $y\in B_\de(\by)$ and each $x\in S_{U'}(y)$, condition \eqref{eq:metric_subregularity} yields
	$\dist_{S_{U'}(\by)}(x)=\dist_{\Phi^{-1}(\by)}(x)\le Ld(y,\by)<L\de<\eps$.
	By \cref{def:upper_semicontinuity}, $S_{U'}$ is upper semicontinuous at $\by$.
\end{proof}

We note that the metric subregularity condition \eqref{eq:metric_subregularity}
from \cref{prop:sufficient_condition_for_quasi_upper_semicontinuity_feasiblity_map} already
amounts to a qualification condition addressing sets of type $\{x\in X\,|\,\by\in\Phi(x)\}$,
see \cite[Section~5]{Gfrerer2013}.
Sufficient conditions for metric subregularity can be found e.g.\ in \cite{BaiYeZhang2019,DontchevRockafellar2014,DontchevGfrererKrugerOutrata2020,Ioffe2017,Kruger2015,Marechal2018,ZhengNg2010}.

We would like to point the reader's attention to the fact that metric subregularity of $\Phi$ is a
quantitative continuity property coming along with a
\emph{modulus} of subregularity $L>0$
while upper semicontinuity of the mappings
$S_{U'}$ in \cref{prop:sufficient_condition_for_quasi_upper_semicontinuity_feasiblity_map} is just a
qualitative continuity property. In this regard, there
exist weaker sufficient conditions ensuring
validity of the upper semicontinuity requirements from \cref{prop:decomposition_lsc}\,\ref{item:lower_semicontinuity_via_decomposition}.
However, it is not clear if such conditions can be easily checked in terms of initial problem data
while this is clearly possible for metric subregularity as the aforementioned list of references underlines.
Finally, we would like to mention that in case where one wants to avoid fixing the component $\bar x\in X$ in the preimage space
in \cref{prop:sufficient_condition_for_quasi_upper_semicontinuity_feasiblity_map}, it is possible to demand that $\Phi^{-1}$
is Lipschitz upper semicontinuous at $\bar y$ in the sense of \cite[p.~10]{KlatteKummer2002}. Again, this is a
quantitative continuity property.

\begin{example}\label{ex:upper_semicontinuity}
	Let $G\colon X\to Y$ be a single-valued mapping between Banach spaces.
	Furthermore, let $C\subset X$ and $K\subset Y$ be nonempty, closed sets.
	We investigate the feasibility mapping $\Phi\colon X\tto Y\times X$ given by
	$\Phi(x):=(G(x)-K,x-C)$ for all $x\in X$ as well as some point $\bx\in X$
	such that $(\bx,(0,0))\in\gph\Phi$ and some neighborhood $U$ of $\bx$.
	Let us define $S\colon Y\times X\tto X$ by means of $S(y,z):=\Phi^{-1}(y,z)\cap U$
	for each pair $(y,z)\in Y\times X$.
	One can check that $S$ is upper semicontinuous at $(0,0)$ if and only if
	\[
		\dist_{K\times C}((G(x_k),x_k))\to 0
		\quad\Longrightarrow\quad
		\lim\limits_{k\to+\infty}\dist_{G^{-1}(K)\cap C}(x_k)=0
	\]
	for each sequence $\{x_k\}_{k\in\N}\subset U$,
	and this is trivially satisfied if $G$ is continuous and $X$ is finite-dimensional.
	For the purpose of completeness, let us also mention that $S$ is partially weakly sequentially upper semicontinuous at $(0,0)$
	if and only if
	\begin{equation}\label{eq:partially_weakly_sequantially_usc_geonetric_constraints}
		x_k\weakly x,\quad \dist_{K\times C}((G(x_k),x_k))\to 0
		\quad\Longrightarrow\quad
		x\in G^{-1}(K)\cap C
	\end{equation}
	is valid for each sequence $\{x_k\}_{k\in\N}\subset U$ and each point $x\in U$.
	Again, this is inherent if $G$ is continuous while $X$ is finite-dimensional and $U$ is closed.
	
	In infinite-dimensional situations, whenever $G$ is continuously Fr\'{e}chet differentiable and
	$C$ as well as $K$ are convex,
	Robinson's constraint qualification, given by
	\[
		G'(\bar x)\left[\bigcup\nolimits_{\alpha\in[0,+\infty)}\alpha(C-\bar x)\right]
		-
		\bigcup\nolimits_{\alpha\in[0,+\infty)}\alpha(K-G(\bar x))
		=
		Y,
	\]
	is equivalent to so-called metric regularity of $\Phi$ at $(\bx,(0,0))$,
	see \cite[Proposition~2.89]{BonnansShapiro2000}, and the latter
	is sufficient for metric subregularity of $\Phi$ at $(\bx,(0,0))$.
\end{example}

The final corollary of this section now follows from
\cref{prop:decomposition_lsc,cor:sufficient_criteria_lower_semicontinuity_relative_to_a_set,prop:sufficient_condition_for_quasi_upper_semicontinuity_feasiblity_map}.
\begin{corollary}\label{cor:lower_semicontinuity_relative_to_svm_decomposition}
	Fix metric spaces $X$ and $Y$, $\Phi\colon X\tto Y$, $\varphi\colon X\to\R_\infty$,  $\by\in Y$, and $\bx\in\Phi^{-1}(\by)\cap\dom\varphi$.
	Assume that $\Phi$ is metrically subregular at $(\bx,\by)$ and that $\varphi$ satisfies one of the
	conditions~\ref{item:lsc_via_interiority}-\ref{item:lsc_via_uniform_continuity} of \cref{cor:sufficient_criteria_lower_semicontinuity_relative_to_a_set}.
	Then $\varphi$ is lower semicontinuous near $\bx$ relative to $\Phi$ at $\by$.	
\end{corollary}

\section{Optimality conditions and approximate stationarity}\label{sec:main_result}

We consider the optimization problem
\begin{equation}\label{eq:basic_problem}\tag{P}
	\min\{\varphi(x)\,|\,\bar y\in\Phi(x)\},
\end{equation}
where $\varphi\colon X\to\R_\infty$ is an arbitrary function,
$\Phi\colon X\rightrightarrows Y$ is a set-valued mapping
{between Banach spaces},
and $\bar y\in\Im \Phi$.
Let us mention that the model \eqref{eq:basic_problem} is quite general and covers numerous
important classes of optimization problems, see e.g.\ \cite{Gfrerer2013,Mehlitz2020b}
for a discussion.
The constrained problem \eqref{eq:basic_problem} is obviously equivalent to the unconditional minimization of
the restriction $\varphi_{\Phi^{-1}(\by)}$ of $\varphi$ to $\Phi^{-1}(\by)$.
We say that $\bx$ is an $\eps$-minimal point of problem \eqref{eq:basic_problem} on $U$
if it is an $\eps$-minimal point of $\varphi_{\Phi^{-1}(\by)}$ on $U$.
Analogously, stationary points of \eqref{eq:basic_problem} are defined.

The next theorem presents dual (i.e., subdifferential/coderivative based) necessary conditions for
$\eps$-minimal points of problem \eqref{eq:basic_problem}.

\begin{theorem}\label{thm:main_result}
Let $X$ and $Y$ be Banach spaces,
$\varphi\colon X\to\R_\infty$ be
lower semicontinuous, $\Phi\colon X\rightrightarrows Y$ have closed graph,
and fix $\bar y\in Y$, $\bx\in\dom \varphi\cap\Phi^{-1}(\by)$, $U\subset X$, $\eps>0$, as well as $\de>0$.
Assume that ${B}_{\de}(\bx)\subset U$, and
\begin{enumerate}
\item
on $U$, $\varphi$ is bounded from below and
lower semicontinuous relative to $\Phi$ at $\by$;
\item\label{item:Aplund_or_convexity}
either $X$ and $Y$ are Asplund, or
$\varphi$ and $\gph\Phi$ are convex.
\end{enumerate}
Suppose that $\bx$ is an $\eps$-minimal point of problem \eqref{eq:basic_problem} on $U$.
Then, for each $\eta>0$,
there exist points
$x_1,x_2\in B_\de(\bx)$
and $y_2\in \Phi(x_2)\cap B_\eta(\by)$
such that
$\norm{x_2-x_1}<\eta$,
$\dist_{\gph\Phi}(x_1,\by)<\eta$,
$\varphi(x_1)<\varphi(\bx)+\eta$,
and
\[
	0\in{\sd}\varphi(x_1)+\Im D^*\Phi(x_2,y_2)+\frac{2\eps}\de\B^*.
\]
Moreover, if $\varphi$ and $\gph\Phi$ are convex, then
$\varphi(x_1)\le \varphi(\bx)$.
\end{theorem}
\begin{proof}
Since $\varphi$ is bounded from below on $U$, and
$\bx$ is an $\eps$-minimal of problem \eqref{eq:basic_problem} on $U$,
there exist numbers $c>0$ and $\eps'\in(0,\eps)$ such that
\begin{subequations}\label{eq:boundedness_and_eps_minimality}
	\begin{alignat}{2}
		\label{item:lsc}
			&\forall x\in U\colon & \varphi(x)&>\varphi(\bx)-c,\\
		\label{thm:main_result-0}
			&\forall x\in \Phi^{-1}(\by)\cap U\colon\quad &\varphi(x)&>\varphi(\bar x)-\eps'.
	\end{alignat}
\end{subequations}
For $\ga>0$ and $\ga_1>0$, let the functions
$\phi_\gamma,\hat\phi_{\gamma,\ga_1}\colon X\times Y\to\R_\infty$
be given by
\begin{subequations}\label{eq:surrogate_functions}
	\begin{align}
		\label{phi}
			\forall (x,y)\in X\times Y\colon\qquad
			\phi_\ga(x,y)&:=
			\varphi(x)+\ga\bigl(\norm{y-\by}+\dist_{\gph\Phi}(x,y)\bigr),
			\\
		\label{phi'}
			\hat\phi_{\gamma,\ga_1}(x,y)&:=
			\phi_\ga(x,y)+\ga_1\norm{x-\bx}^2.
	\end{align}
\end{subequations}
Set $\de_0:=\de\eps'/\eps$, and choose numbers $\de'\in(\de_0,\de)$ and $\xi\in(0,\de-\de')$ such that $\xi(\de'+2)<2(\eps\de'/\de-\eps')$.
Fix an arbitrary $\eta>0$ and a positive number $\eta'<\min(\eta,2(\de-\de'))$.
Set
$\ga_1:=(\eps'+\xi)/(\de')^2$.
Observe that $\hat\phi_{\gamma,\ga_1}(\bx,\by)=\varphi(\bx)$, and $\hat\phi_{\gamma,\ga_1}$ is bounded from below on $\overline{B}_{\de'}(\bar x)\times Y$ due
to \eqref{item:lsc}.
By Ekeland's variational principle, see \cref{lem:Ekeland},
for each $k\in\N$, there exists a point
$(x_k,y_k)\in\overline{B}_{\de'}(\bx)\times Y$ such that
\begin{subequations}\label{eq:Ekeland_consequences}
	\begin{align}
		\label{eq:Ekeland_1}
			&\hat\phi_{k,\ga_1}(x_k,y_k)\le \varphi(\bx),\\
		\label{eq:Ekeland_2}
	\forall (x,y)\in\overline{B}_{\de'}(\bar x)\times Y\colon\quad
	&\hat\phi_{k,\ga_1}(x,y)+\xi\norm{(x,y)-(x_k,y_k)}
	\ge
	\hat\phi_{k,\ga_1}(x_k,y_k).
	\end{align}
\end{subequations}
It follows from \eqref{item:lsc}, \eqref{eq:surrogate_functions}, and \eqref{eq:Ekeland_1} that
\begin{equation*}
k\bigl(\norm{y_k-\by}+\dist_{\gph\Phi}(x_k,y_k)\bigr)
	+\ga_1\norm{x_k-\bx}^2
	\le \varphi(\bx)-\varphi(x_k)<c,
\end{equation*}
and, consequently,
\begin{subequations}\label{eq:Ekeland_consequences_revisited}
	\begin{align}
	\label{eq:Ekeland_3}
	\norm{y_k-\by}+\dist_{\gph\Phi}(x_k,y_k)<c/k,
	\\
	\label{eq:Ekeland_4}
	\ga_1\norm{x_k-\bx}^2\le \varphi(\bx)-\varphi(x_k)
	\end{align}
\end{subequations}
are valid for all $k\in\N$.
By \eqref{eq:Ekeland_3}, $y_k\to\by$ and $\dist_{\gph\Phi}(x_k,y_k)\to0$ as $k\to+\infty$,
and $y_k\in B_{\eta'/4}(\by)$ as well as $\dist_{\gph\Phi}(x_k,y_k)<\eta'/4$ follow for all $k>4c/\eta'$.
Recall that $\{x_k\}_{k\in\N}
+\rho\B\subset B_\de(\bx)\subset U$ for any positive $\rho<\de-\de'$.
By \cref{prop:sequential_characterization_lsc}, there exist an integer $\bar k>4c/\eta'$
and a point $x'\in\Phi^{-1}(\by)\cap U$ such that $\varphi(x')<\varphi(x_{\bar k})+\xi$.
By \eqref{thm:main_result-0}, we have
$\varphi(\bx)-\varphi(x')<\eps'$.
Set $\ga:=\bar k$, $\hat x:=x_{\bar k}$, and $\hat y:=y_{\bar k}$.
Thus, $\hat y\in B_{\eta'/4}(\by)$ and
$\dist_{\gph\Phi}(\hat x,\hat y)<\eta'/4$.
By \eqref{eq:Ekeland_4},
\begin{align*}
\ga_1\norm{\hat x-\bx}^2
\le (\varphi(\bx)-\varphi(x')) +(\varphi(x')-\varphi(\hat x))
<\eps'+\xi=\ga_1(\de')^2.
\end{align*}
Hence, we find $\norm{\hat x-\bx}<\de'$.
In view of \eqref{phi'},
condition \eqref{eq:Ekeland_1} is equivalent to
\begin{equation}
\label{eq:intermediate_estimate:main_result}
\phi_\ga(\hat x,\hat y)+\ga_1\norm{\hat x-\bx}^2\le \varphi(\bx).
\end{equation}

For each $(x,y)\in B_{\de'}(\bar x)\times Y$ different from $(\hat x,\hat y)$,
it follows from \eqref{phi'} that
\begin{align*}
\frac{\phi_\ga(\hat x,\hat y)-\phi_\ga(x,y)}
{\norm{(x,y)-(\hat x,\hat y)}}
&=
\frac{\hat\phi_{\ga,\ga_1}(\hat x,\hat y)-\hat\phi_{\ga,\ga_1}(x,y)
	+\ga_1\bigl(\norm{x-\bx}^2-\norm{\hat x-\bx}^2\bigr)}
	{\norm{(x,y)-(\hat x,\hat y)}}
\\
&\le
\frac{\hat\phi_{\ga,\ga_1}(\hat x,\hat y)-\hat\phi_{\ga,\ga_1}(x,y)
	+\ga_1\norm{x-\hat x}(\norm{x-\bx}+\norm{\hat x-\bx})}
	{\norm{(x,y)-(\hat x,\hat y)}}
\\
&\le
\frac{\hat\phi_{\ga,\ga_1}(\hat x,\hat y)-\hat\phi_{\ga,\ga_1}(x,y)}
{\norm{(x,y)-(\hat x,\hat y)}}
	+\ga_1\bigl(\norm{x-\bx}+\norm{\hat x-\bx}\bigr),
\end{align*}
and, consequently, in view of \eqref{eq:Ekeland_2},
\begin{align*}
\sup_{(x,y)\in(B_{\de'}(\bar x)\times Y)
\setminus\{(\hat x,\hat y)\}}
\frac{\phi_\ga(\hat x,\hat y)-\phi_\ga(x,y)}
{\norm{(x,y)-(\hat x,\hat y)}}
<\xi+2\ga_1\de'=\frac{2\eps'+\xi(\de'+2)}\de'
<\frac{2\eps}\de.
\end{align*}
Since $\hat x$ is an interior point of $\overline B_{\de'}(\bx)$,
it follows that
\begin{align}\label{eq:main_result_local_slope_condition}
\limsup_{(x,y)\to(\hat x,\hat y)}
\frac{\phi_\ga(\hat x,\hat y)-\phi_\ga(x,y)}
{\norm{(x,y)-(\hat x,\hat y)}}<\frac{2\eps}\de.
\end{align}

By \eqref{phi} and \eqref{eq:intermediate_estimate:main_result}, we find $\varphi(\hat x)\le \varphi(\bx)$, and due to
\eqref{eq:main_result_local_slope_condition}, there is a number
$\hat\eps\in(0,\frac{2\eps}\de)$ such that
\begin{align*}
	\liminf\limits_{(x,y)\to(\hat x,\hat y)}
	\frac{\phi_\gamma(x,y)+\hat\eps\norm{(x,y)-(\hat x,\hat y)}-\phi_{\gamma}(\hat x,\hat y)}
	{\norm{(x,y)-(\hat x,\hat y)}}
	\geq 0.
\end{align*}
Set $\xi':=2{\eps}/\de-{\hat\eps}>0$.
By definition of the Fr\'{e}chet subdifferential,
the above inequality yields
\begin{align}\label{eq:main_result_Fermat_rule}
(0,0)\in{\partial} \left(\phi_\ga+\hat\eps\norm{(\cdot,\cdot)}\right)
(\hat x,\hat y).
\end{align}

Condition \eqref{eq:main_result_Fermat_rule} can be rewritten as
$(0,0)\in{\partial}\left(\varphi+\ga g+h\right)(\hat x,\hat y)$, where the functions
$g,h\colon X\times Y\to\R$ are given by
\begin{align*}
\forall (x,y)\in X\times Y\colon\quad
g(x,y)&:=\dist_{\gph\Phi}(x,y),\\
h(x,y)&:=\ga\|y-\bar y\|+\hat\eps\|(x,y)-(\hat x,\hat y)\|.
\end{align*}
Note that $g$ and $h$ are Lipschitz continuous, and $h$ is convex.
We distinguish two cases.

\underline{Case 1: Let $X$ and $Y$ be Asplund spaces.}
Let us recall the estimates $\norm{\hat x-\bar x}<\de'<\delta$, $\norm{\hat y-\bar y}<\eta'/4<\eta/4$,
$\dist_{\gph\Phi}(\hat x,\hat y)<\eta'/4<\eta/4$, and $\varphi(\hat x)\leq\varphi(\bar x)$.
By the fuzzy sum rule, see \cref{lem:SR}\,\ref{item:fuzzy_sum_rule},
there exist points
$(x_1,y_1),(u_2,v_2)\in X\times Y$
arbitrarily close to $(\hat x,\hat y)$
with $\varphi(x_1)$ arbitrarily close to $\varphi(\hat x)$,
so that
\begin{gather*}
\norm{x_1-\bx}<\de,\quad
\norm{u_2-\bx}<\de',\quad
\varphi(x_1)<\varphi(\bx)+\eta,\quad
\norm{y_1-\by}<\frac{\eta}2,
\\
\norm{v_2-\by}<\frac{\eta'}2,\quad
\norm{u_2-x_1}<\frac{\eta'}2,\quad
\dist_{\gph\Phi}(x_1,y_1)<\frac{\eta}2,\quad
\dist_{\gph\Phi}(u_2,v_2)<\frac{\eta'}4,
\end{gather*}
and subgradients $x_{1}^*\in{\partial}\varphi(x_{1})$
and $(u_{2}^*,v_{2}^*)\in{\partial}g(u_{2},v_{2})$ satisfying
$$\norm{x_{1}^*+\ga u_{2}^*}<\hat\eps+\frac{\xi'}2.$$
Thus, $x_1\in B_\de(\bx)$
and
$\dist_{\gph\Phi}(x_1,\by)<\dist_{\gph\Phi}(x_1,y_1) +\norm{y_1-\by}<\eta$.
In view of
\cref{lem:subdifferential_distance_function}\,\ref{item:sdf_distance_function_out_of_set_points},
there exist
$(x_{2},y_{2})\in\gph\Phi$ and
$(u_{2}^{*\prime},v_{2}^{*\prime})\in { N}_{\gph\Phi}(x_{2},y_{2})$ such that
\begin{gather*}
\norm{(x_2,y_2)-(u_2,v_2)}
<\dist_{\gph\Phi}(u_2,v_2)+\frac{\eta'}{4}<\frac{\eta'}2,\quad
\norm{u_{2}^{*\prime}-u_{2}^*}<\frac{\xi'}{2\ga}.
\end{gather*}
Set $x_2^*:=\ga u_{2}^{*\prime}$ and $y^*:=-\ga v_{2}^{*\prime}$.
Thus, $x_2^*\in D^*\Phi(x_2,y_2)(y^*)$, and we have
\begin{align*}
\norm{y_2-\by}\le&\norm{v_2-\by}+\norm{y_2-v_2}<\eta'<\eta,
\\
\norm{x_2-\bx}\le&\norm{u_2-\bx}+\norm{x_2-u_2}
<\de'+\frac{\eta'}{2}<\de,
\\
\norm{x_2-x_1}\le&\norm{x_2-u_2}+\norm{u_2-x_1}<\eta'<\eta,
\\
\norm{x_1^*+x_2^{*}}
\leq&	\norm{x_1^*+\gamma\,u_2^*}+\gamma\norm{u_2^{*\prime}-u_2^*}
<\hat\eps+{\xi'}=\frac{2\eps}\de.
\end{align*}

\underline{Case 2:
Let $\varphi$ and $\gph\Phi$ be convex.}
We have $\hat x\in B_{\de'}(\bx)\subset B_\de(\bx)$,
$\varphi(\hat x)\le \varphi(\bx)$,
$\norm{\hat y-\bar y}<\eta'/4$,
and
$\dist_{\gph\Phi}(\hat x,\hat y)<\eta'/4<\eta$.
By the convex sum rule, see \cref{lem:SR}\,\ref{item:convex_sum_rule},
there exist subgradients
$x_{1}^*\in{\partial}\varphi(\hat x)$
and $(u_{2}^*,v_{2}^*)\in{\partial}g(\hat x,\hat y)$ satisfying
\begin{align*}
	\norm{x_{1}^*+\ga u_{2}^*}\le\hat\eps.
\end{align*}
In view of
\cref{lem:subdifferential_distance_function}\,\ref{item:sdf_distance_function_out_of_set_points},
there exist
$(x_{2},y_{2})\in\gph\Phi$ and
$(u_{2}^{*\prime},v_{2}^{*\prime})\in { N}_{\gph\Phi}(x_{2},y_{2})$ such that
\begin{align*}
\norm{(x_2,y_2)-(\hat x,\hat y)}
<\dist_{\gph\Phi}(\hat x,\hat y)+\frac{\eta'}{4},\quad
	\norm{u_{2}^{*\prime}-u_{2}^*}
<\frac{\xi'}{\ga}.
\end{align*}
Set $x_1:=\hat x$, $x_2^*:=\ga u_{2}^{*\prime}$, and $y^*:=-\ga v_{2}^{*\prime}$.
Thus, $x_1\in B_\de(\bx)$,
$\dist_{\gph\Phi}(x_1,\by)<\eta'/2<\eta$,
$\varphi(x_1)\leq \varphi(\bar x)$, and
$x_2^*\in D^*\Phi(x_2,y_2)(y^*)$.
Replacing $(u_2,v_2)$ with $(\hat x,\hat y)$ in the corresponding estimates established in Case 1, we obtain
\begin{align*}
\norm{y_2-\by}<\eta,\quad
\norm{x_2-\bx}<\de,\quad
\norm{x_2-x_1}<\eta,
\\
\norm{x_1^*+x_2^{*}}
\leq\norm{x_1^*+\gamma\,u_2^*}
+\gamma\norm{u_2^{*\prime}-u_2^*}
<\hat\eps+\xi'=\frac{2\eps}\de.
\end{align*}
This completes the proof.
\end{proof}

Clearly, \cref{thm:main_result} provides dual necessary conditions for
$\eps$-minimality of a feasible point of problem \eqref{eq:basic_problem}
under some additional structural assumptions on the data which are almost for free in the finite-dimensional
setting, see \cref{cor:lower_semicontinuity_svm_via_weak_sequential_lower_semicontinuity},
and of reasonable strength in the infinite-dimensional one.
In the subsequent remark, we comment on additional primal and dual conditions for $\eps$-minimality
which can be distilled from the proof of \cref{thm:main_result}.
\begin{remark}
\label{rem:primal_conditions}
\begin{enumerate}
\item\label{item:primal_conditions_in_metric_spaces}
In the proof of \cref{thm:main_result}, more sets of necessary conditions for local
$\eps$-minimality of a feasible point of problem \eqref{eq:basic_problem}
have been established along the way.
Moreover, the first part of the proof does not use the linear structure of the spaces,
i.e., the arguments work in the setting of general complete metric spaces $X$ and $Y$.
The conditions can be of independent interest and are listed below.
We assume that $X$ and $Y$ are complete metric spaces and all the other assumptions of \cref{thm:main_result} are satisfied,
except condition~\ref{item:Aplund_or_convexity}.
\medskip

\textbf{Necessary conditions for local $\eps$-minimality.}
There is a $\de_0\in(0,\de)$ such that,
for each $\de'\in(\de_0,\de)$ and $\eta>0$,
there exist points $\hat x\in B_{\de'}(\bx)$ and
$\hat y\in B_\eta(\by)$ satisfying
$\dist_{\gph\Phi}(\hat x,\hat y)<\eta$,
and numbers $\ga>0$ and $\ga_1>0$
such that,
with the function
$\phi_\gamma\colon X\times Y\to\R_\infty$
given by
\begin{align*}
\forall (x,y)\in X\times Y\colon\quad
\phi_\ga(x,y):=
\varphi(x)+\ga\bigl(d(y,\bar y)+\dist_{\gph\Phi}(x,y)\bigr),
\end{align*}
the following conditions hold:
\begin{itemize}
\item
$\phi_\ga(\hat x,\hat y)+\ga_1 d(\hat x,\bar x)^2\le \varphi(\bx)$,
and

\item
\underline{primal nonlocal condition \textup{\textbf{(PNLC)}}}:
$\sup\limits_{\substack{(x,y)\ne(\hat x,\hat y)\\
x\in B_{\de'}(\bar x)}}
		\dfrac{\phi_\ga(\hat x,\hat y)-\phi_\ga(x,y)}{d((x,y),(\hat x,\hat y))}
	<
	\dfrac{2\eps}\de$,

\item
\underline{primal local condition \textup{\textbf{(PLC)}}}:
$\limsup\limits_{(x,y)\to(\hat x,\hat y)}
\dfrac{\phi_\ga(\hat x,\hat y)-\phi_\ga(x,y)}
{d((x,y),(\hat x,\hat y))}<\dfrac{2\eps}\de$,

\item
\underline{dual condition \textup{\textbf{(DC)}}} ($X$ and $Y$ are Banach spaces):
condition \eqref{eq:main_result_Fermat_rule} is satisfied with some $\hat\eps\in(0,\frac{2\eps}\de)$.
\end{itemize}

The relationship between the conditions is as follows:
\textup{\textbf{(PNLC)}}
$\Rightarrow$
\textup{\textbf{(PLC)}}
$\Rightarrow$
\textup{\textbf{(DC)}}.
The dual conditions in \cref{thm:main_result} are consequences of the above conditions.

Let us note that the left-hand side\ in \textup{\textbf{(PNLC)}} is the \emph{nonlocal slope}, see \cite{FabianHenrionKrugerOutrata2010},
of the function $\phi_\ga+i_{B_{\de'}(\bx)}$ at $(\hat x,\hat y)$,
while the left-hand side\ in \textup{\textbf{(PLC)}} is the conventional slope of $\phi_\ga$ at $(\hat x,\hat y)$.

\item
Since the function $\varphi$ in \cref{thm:main_result} is assumed to be lower semicontinuous,
it is automatically bounded from below on \emph{some} neighborhood of $\bx$.
We emphasize that \cref{thm:main_result} requires all the conditions to hold on the same
set $U$ containing a neighborhood of $\bx$.
\end{enumerate}
\end{remark}

As a consequence of \cref{thm:main_result}, we obtain necessary conditions
characterizing local minimizers of \eqref{eq:basic_problem}.

\begin{corollary}\label{cor:necessary_conditions_local_minimizer}
Let $X$ and $Y$ be Banach spaces,
$\varphi\colon X\to\R_\infty$
lower semicontinuous, $\Phi\colon X\rightrightarrows Y$ have closed graph, $\bar y\in Y$, and
$\bx\in\dom \varphi\cap\Phi^{-1}(\by)$.
Assume that
\begin{enumerate}
\item
the function $\varphi$ is lower semicontinuous near $\bx$ relative to $\Phi$ at $\by$;
\item
either $X$ and $Y$ are Asplund, or
$\varphi$ and $\gph\Phi$ are convex.
\end{enumerate}
Suppose that
$\bx$ is a local minimizer
of \eqref{eq:basic_problem}.
Then,
for each $\eps>0$,
there exist points
$x_1,x_2\in B_\eps(\bx)$ and
$y_2\in\Phi(x_2)\cap B_{\eps}(\by)$ such that
$|\varphi(x_1)-\varphi(\bx)|<\eps$ and
\[
	0\in{\sd}\varphi (x_1)+\Im D^*\Phi(x_2,y_2)+\eps\B^*.
\]
Moreover, if $\varphi$ and $\gph\Phi$ are convex, then
$\varphi(x_1)\le \varphi(\bx)$.
\end{corollary}

\begin{proof}
Let a number $\eps>0$ be arbitrarily chosen.
Set $\eps':=\eps/2$.
By the assumptions and \cref{rem:minimality_vs_stationarity},
there exists a $\de\in(0,\varepsilon)$ such that
on $U:=\overline{B}_{\de}(\bx)$ the function $\varphi$ is bounded from below and
lower semicontinuous relative to $\Phi$ at $\by$,
and $\bx$ is an $\eps'\de$-minimal point of $\varphi_{\Phi^{-1}(\by)}$ on $U$.
Thus, all the assumptions of \cref{thm:main_result} are satisfied.
Moreover, $2(\eps'\de)/\de=\eps$ and, since $\varphi$ is lower semicontinuous,
one can ensure that $\varphi(x_1)>\varphi(\bx)+\eps$.
Hence, taking any $\eta\in(0,\eps)$, the assertion follows from \cref{thm:main_result}.
\end{proof}

In the subsequent remark, we comment on the findings in \cref{cor:necessary_conditions_local_minimizer}.
\begin{remark}\label{rem:asymptotic_stationarity_local_minimizers}
\begin{enumerate}
	\item The analogues of the necessary conditions in \cref{rem:primal_conditions}\,\ref{item:primal_conditions_in_metric_spaces}
		are valid in the setting of \cref{cor:necessary_conditions_local_minimizer}, too.
		More precisely, it suffices to replace $\frac{2\eps}\de$ with just $\eps$ in the involved conditions.
	\item The necessary conditions in \cref{cor:necessary_conditions_local_minimizer} hold for each stationary point
		(not necessarily a local minimizer) of problem \eqref{eq:basic_problem}.
\end{enumerate}
\end{remark}

We now consider an important particular case of problem \eqref{eq:basic_problem}, namely
\begin{equation}\label{eq:constrained_problem}\tag{$\widetilde{\text{P}}$}
	\min\{\varphi(x)\,|\,x\in\Omega\},
\end{equation}
where $\Omega\subset X$ is a nonempty
subset of a Banach space.
To obtain this setting from the
one in \eqref{eq:basic_problem},
it suffices to consider
the set-valued mapping $\Phi\colon X\rightrightarrows Y$ whose graph is given by
$\gph\Phi:=\Omega\times Y$.
Here, $Y$ can be an arbitrary Asplund space, e.g., one can take $Y:=\R$.
Observe that $\Phi^{-1}(y)=\Omega$ holds for all $y\in Y$.
Hence, by \cref{def:upper_semicontinuity}, for all $y\in Y$, the mapping $\Phi^{-1}$ is upper semicontinuous at $y$.
Moreover,
$N_{\gph\Phi}(x,y)=N_\Omega(x)\times\{0\}$.
Thus, the next statement is a consequence of \cref{prop:decomposition_lsc} and \cref{thm:main_result}.

\begin{theorem}\label{thm:main_result_wrt_set}
Let $X$ be a Banach space,
$\varphi\colon X\to\R_\infty$ lower semicontinuous,
$\Omega\subset X$ a closed set,
and fix $\bx\in\dom \varphi\cap\Omega$, $U\subset X$, $\eps>0$, and $\de>0$.
Assume that ${B}_{\de}(\bx)\subset U$, and
\begin{enumerate}
\item
on $U$, the function
$\varphi$ is bounded from below and
lower semicontinuous relative to $\Omega$;
\item
either $X$ is Asplund, or
$\varphi$ and $\Omega$ are convex.
\end{enumerate}
Suppose that
$\bx$ is an $\eps$-minimal point of problem \eqref{eq:constrained_problem} on $U$.
Then, for each $\eta>0$, there exist points
$x_1\in B_\de(\bx)$
and
$x_2\in\Omega\cap B_{\de}(\bx)$ such that $\norm{x_2-x_1}<\eta$,
$\varphi(x_1)<\varphi(\bx)+\eta$,
and
\[
	0\in{\sd}\varphi(x_1)+ N_\Omega(x_2)+\frac{2\eps}\de\B^*.
\]
Moreover, if $\varphi$ and $\Omega$ are convex, then
$\varphi(x_1)\le \varphi(\bx)$.
\end{theorem}

The next corollary follows immediately.

\begin{corollary}\label{cor:consequence_for_set_constrained_problems}
Let $X$ be a Banach space,
$\varphi\colon X\to\R_\infty$ lower semicontinuous,
$\Omega\subset X$ a closed set,
and
$\bx\in\dom \varphi\cap\Omega$.
Assume that
\begin{enumerate}
\item
the function
$\varphi$ is
lower semicontinuous near $\bx$ relative to $\Omega$;
\item
either $X$ is Asplund, or
$\varphi$ and $\Omega$ are convex.
\end{enumerate}
Suppose that
$\bx$ is a local minimizer of \eqref{eq:constrained_problem}.
Then, for each $\varepsilon>0$, there exist
$x_1\in B_\eps(\bx)$ and
$x_2\in\Omega\cap B_{\eps}(\bx)$ such that
$|\varphi(x_1)-\varphi(\bx)|<\eps$ and
\[
	0\in{\sd}\varphi (x_1)+ N_\Omega(x_2)+\eps\B^*.
\]
Moreover, if $\varphi$ and $\Omega$ are convex, then
$\varphi(x_1)\le \varphi(\bx)$.
\end{corollary}

Whenever $\varphi$ is Lipschitz continuous around $\bar x$, the assertion of
\cref{cor:consequence_for_set_constrained_problems} is an immediate consequence
of Fermat's rule and the sum
rules stated in \cref{lem:SR}.
We note that \cref{cor:consequence_for_set_constrained_problems}
is applicable in more general situations, exemplary, if $\varphi$ is only
uniformly continuous in a neighborhood of the investigated local minimizer,
see \cref{cor:sufficient_criteria_lower_semicontinuity_relative_to_a_set},
or if $X$ is finite-dimensional, see \cref{cor:lower_semicontinuity_near_point_set_via_weak_sequential_lower_semicontinuity}.

Note that the dual necessary optimality conditions in \cref{cor:necessary_conditions_local_minimizer,cor:consequence_for_set_constrained_problems}
do not hold at the reference point
but at some other points arbitrarily close to it.
Such conditions describe certain properties of admissible points which can be interpreted as a kind of dual approximate stationarity.
The precise meaning of approximate stationarity will be discussed in \cref{sec:asymptotic_stuff} in the setting of
geometrically-constrained optimization problems.

\section{Generalized separation and extremal principle}\label{sec:generalized_separation}

Below, we discuss certain generalized extremality and separation properties
of a collection of closed subsets $\Omega_1,\ldots,\Omega_n$ of a Banach space $X$, having a common point $\bx\in\bigcap_{i=1}^n\Omega_i$.
Here, $n$ is an integer satisfying $n\geq 2$.
We write $\{\Omega_1,\ldots,\Omega_n\}$ to denote the collection of sets as a single object.

We begin with deriving necessary conditions for so-called
$\mathcal{F}$-extremality of a collection of sets.
The property in the definition below is determined by a nonempty family $\mathcal{F}$ of nonnegative lower semicontinuous functions
$f\colon X^{n}\to\R_\infty$ and
mimics the corresponding conventional one, see e.g.\ \cite{KruMor80}.

\begin{definition}\label{def:extremality}
Let a family $\mathcal{F}$ described above be given.
Suppose that, for each $f\in\mathcal{F}$, the function
$\hat f\colon X^{n}\to\R_\infty$ is defined by
\begin{equation}\label{eq:hatf}
\forall z:=(x_1,\ldots,x_n)\in X^{n}\colon\quad
\hat f(z):=f(x_1-x_n,\ldots,x_{n-1}-x_n,x_n).
\end{equation}
The collection $\{\Omega_1,\ldots,\Omega_n\}$ is \emph{$\mathcal{F}$-extremal at $\bx$} if,
for each $\eps>0$, there exist a function $f\in\mathcal{F}$ and a number $\rho>0$
such that $f(0,\ldots,0,\bx)<\eps$ and
\begin{equation}\label{eq:extremality_nonnegativity}
\forall
x_i\in\Omega_i+\rho\B\; (i=1,\ldots,n)\colon\quad
\hat f(x_1,\ldots,x_n)>0.
\end{equation}
\end{definition}

The following theorem, which is based on \cref{thm:main_result_wrt_set}, provides
a general necessary condition for $\mathcal F$-extremality.
\begin{theorem}
\label{thm:necessary_condition_extremality}
Assume that
\begin{enumerate}
\item
there is a neighborhood $U$ of $\bx$ such that,
for each $f\in\mathcal{F}$, the function
$\hat f\colon X^{n}\to\R_\infty$ defined by
\eqref{eq:hatf} is lower semicontinuous on $U^n$ relative to $\Omega:=\Omega_1\times\ldots\times\Omega_n$;
\item
either $X$ is Asplund, or
$\Omega_1,\ldots,\Omega_n$ and all $f\in\mathcal{F}$ are convex.
\end{enumerate}
Suppose that
the collection $\{\Omega_1,\ldots,\Omega_n\}$ is
$\mathcal{F}$-extremal at $\bx$.
Then,
for each $\eps>0$ and $\eta>0$, there exist a function $f\in\mathcal{F}$ with $f(0,\ldots,0,\bx)<\eps$ and points
$x_i\in\Omega_i\cap B_{\eps}(\bx)$, $x'_i\in B_{\eta}(x_i)$, and $x_i^*\in X^*$ $(i=1,\ldots,n)$ such that
\begin{subequations}\label{eq:NC_extremality}
	\begin{align}
	\label{eq:NC_extremality_normals_close_to_cone}
	\sum_{i=1}^{n} \dist_{N_{\Omega_i}(x_i)}\left(x_i^*\right) <\eps,
	\\
	\label{eq:NC_extremality_bounds_function_value}
	0<f(w)<f(0,\ldots,0,\bx)+\eta,
	\\
	\label{eq:NC_extremality_subdifferential}
	-\left(x_{1}^*,\ldots,x_{n-1}^*,\sum_{i=1}^{n}x_i^*\right)\in\sd f(w),
\end{align}
\end{subequations}
where $w:=(x'_1-x'_n,\ldots,x'_{n-1}-x'_n,x'_n)\in X^{n}$.
Moreover, if $f$ and $\Omega_1,\ldots,\Omega_n$ are convex, then
$f(w)\le f(0,\ldots,0,\bx)$.
\end{theorem}

\begin{proof}
Let arbitrary numbers
$\eps>0$ and $\eta>0$ be fixed.
Choose a number $\de\in(0,\eps)$ so that ${B}_{\de}(\bx)\subset U$, and set $\eps':=\eps\min(\de/2,1)$.
By \cref{def:extremality}, there exist a function $f\in\mathcal{F}$ and a number $\rho>0$ such that $f(0,\ldots,0,\bx)<\eps'\le\eps$, and
condition \eqref{eq:extremality_nonnegativity} holds, where the function
$\hat f\colon X^{n}\to\R_\infty$ is defined by
\eqref{eq:hatf}.
Observe that $\Omega$ is a
closed subset of the Banach space $X^n$,
$\bz:=(\bx,\ldots,\bx)\in\Omega$, and $\hat f(\bz)=f(0,\ldots,0,\bx)<\eps'$.
Since the function $f$ is nonnegative, so is $\hat f$, and, consequently,
$\bz$ is an $\eps'$-minimal
point of $\hat f_\Omega$
(as well as $\hat f$) on $X^n$.
Set $\eta':=\min(\eta,\rho)$.
By \cref{thm:main_result_wrt_set}, there exist points
$z:=(x_1,\ldots,x_n)\in\Omega\cap B_{\de}(\bz)$,
$z':=(x_1',\ldots,x_n')\in B_{\eta'}(z)$,
and
$x^*:=(x_{1}^*,\ldots, x_{n}^*)\in(X^*)^n$ such that
$f(w)=\hat f(z')<\hat f(\bz)+\eta=f(0,\ldots,0,\bx)+\eta$, and
\begin{equation}\label{eq:conditions_for_hatf}
-x^*\in\sd\hat f(z'),\qquad
\dist_{N_\Omega(z)}(x^*)<\frac{2\eps'}\de\le\eps.
\end{equation}
Moreover, if $f$ and $\Omega_1,\ldots,\Omega_n$ are convex, then
$f(w)\le f(0,\ldots,0,\bx)$.
Observe that $x_i'\in\Omega_i+\rho\B$ ($i=1,\ldots,n$), and it follows from \eqref{eq:extremality_nonnegativity} that $f(w)=\hat f(z')>0$
which shows \eqref{eq:NC_extremality_bounds_function_value}.

The function $\hat f$ given by \eqref{eq:hatf} is a composition of $f$ and
the continuous linear mapping $A\colon X^{n}\to X^{n}$ given as follows:
\begin{align*}
\forall (u_1,\ldots,u_n)\in X^{n}\colon\quad
A(u_1,\ldots,u_n):=(u_1-u_n,\ldots,u_{n-1}-u_n,u_n).
\end{align*}
The mapping $A$ is obviously a bijection.
It is easy to check that the adjoint mapping $A^*\colon (X^*)^{n}\to (X^*)^{n}$ is of the form
\begin{equation}\label{eq:def_linear_operator}
\forall (u_1^*,\ldots,u_{n}^*)\in(X^*)^{n}\colon\quad
A^*(u_1^*,\ldots,u_{n}^*):=
\left(u^*_1,\ldots,u^*_{n-1},u_{n}^*-\sum_{i=1}^{n-1}u_{i}^*\right).
\end{equation}
By the Fr\'{e}chet subdifferential chain rule, which can be distilled from
\cite[Theorem~1.66, Proposition~1.84]{Mordukhovich2006}),
we obtain $\sd\hat f(z')=A^*\sd f(w)$,
where
$w=Az'=(x'_1-x'_n,\ldots,x'_{n-1}-x'_n,x'_n)$.
In view of \eqref{eq:def_linear_operator}, the inclusion in \eqref{eq:conditions_for_hatf} is equivalent to \eqref{eq:NC_extremality_subdifferential}.
It now suffices to observe that
$N_\Omega(z)=N_{\Omega_1}(x_1)\times\ldots\times N_{\Omega_n}(x_n)$, and, consequently, the inequality in \eqref{eq:conditions_for_hatf} yields \eqref{eq:NC_extremality_normals_close_to_cone}.
\end{proof}

For the conclusions of \cref{thm:necessary_condition_extremality} to be non-trivial,
one must ensure that the family $\mathcal{F}$
satisfies the following conditions:
\begin{enumerate}
\item\label{item:non_triviality_inf_value}
$\inf_{f\in\mathcal{F}}f(0,\ldots,0,\bx)=0$;
\item\label{item:non_triviality_subgradients}
$\liminf\limits_{w\to(0,\ldots,0,\bar x),\,f\in\mathcal F,\,f(w)\downarrow0,\,w^*\in\sd f(w)}\norm{w^*}>0$.
\end{enumerate}

A typical example of such a family is given by the collection $\mathcal{F}_A$ of
functions of type
\begin{equation}\label{eq:f_standard_extremality}
\forall z:=(x_1,\ldots,x_n)\in X^{n}\colon\quad
f_a(z):=\max_{1\le i\le n-1}\|x_i-a_i\|,
\end{equation}
where $a:=(a_1,\ldots,a_{n-1})\in X^{n-1}$.
The proofs of the conventional extremal principle and its extensions usually employ such functions.
Note that
functions from $\mathcal{F}_A$ are
constant in the last variable.

It is easy to see that, for each $f_a\in\mathcal{F}_A$ and
$z:=(x_1,\ldots,x_n)\in X^n$, the value $f_a(z)$ is the maximum norm of
$(x_1-a_1,\ldots,x_{n-1}-a_{n-1})$ in $X^{n-1}$.
Thus, $f_a(z)>0$ if and only if $(x_1,\ldots,x_{n-1})\ne a$, and
\begin{align*}
	f_a(0,\ldots,0,\bar x)
	=
	\max_{1\le i\le n-1}\|a_i\|
	\to
	0	
	\quad
	\text{as}
	\quad
	a\to0
\end{align*}
showing~\ref{item:non_triviality_inf_value}.
Moreover, $\sd f_a(z)\ne\varnothing$ for all $z\in X^{n}$ and, if $f_a(z)>0$, then $\norm{w^*}=1$ for all
$w^*\in\sd f_a(w)$, i.e., the
limit
in~\ref{item:non_triviality_subgradients} equals $1$.
Observe also that, since each function $f_a\in\mathcal{F}_A$ is convex and Lipschitz continuous,
the same holds true for the corresponding function $\hat f_a$ defined by \eqref{eq:hatf}.
Hence, $\hat f_a$ is automatically lower semicontinuous near each point of $X^n$
relative to each set containing this point, see \cref{cor:sufficient_criteria_lower_semicontinuity_relative_to_a_set}.

When $f_a\in\mathcal{F}_A$ is given by \eqref{eq:f_standard_extremality}, condition \eqref{eq:extremality_nonnegativity} takes the following form:
\begin{equation}\label{eq:separation_new}
\bigcap_{i=1}^{n-1}(\Omega_i+\rho\B-a_i)\cap (\Omega_n+\rho\B)=\varnothing.
\end{equation}
With this in mind, the extremality property in \cref{def:extremality} admits a geometric interpretation.

\begin{proposition}\label{prop:geometric_version_extremality}
The collection $\{\Omega_1,\ldots,\Omega_n\}$ is
$\mathcal{F}_A$-extremal at $\bx$ if and only if,
for each $\eps>0$, there exist vectors $a_1,\ldots,a_{n-1}\in X$ and a number $\rho>0$ such that
$\max_{1\le i\le n-1}\|a_i\|<\eps$, and
condition \eqref{eq:separation_new} holds.
\end{proposition}

The characterization in \cref{prop:geometric_version_extremality} means that sets with nonempty intersection can be
``pushed apart'' by arbitrarily small translations in such a way that even small enlargements
of the sets become nonintersecting.
Note that condition \eqref{eq:separation_new} is
stronger
than the conventional extremality property originating from \cite{KruMor80},
which corresponds to setting $\rho=0$ in \eqref{eq:separation_new}.
The converse statement is not true as the
next example shows.

\begin{example}\label{ex:extremality_more_restrictive}
Consider the closed sets $\Omega_1,\Omega_2\subset\R^2$ given by
\begin{align*}
	\Omega_1:=\left\{(x,y)\mid x\ge0,\;y=0\right\},\quad
	\Omega_2:=\left\{(x,y)\mid x\ge 0,\;|y|\ge e^{-x} \right\}\cup\{(0,0)\},
\end{align*}
see \cref{fig:counterexample_extremality}.
We have $\Omega_1\cap\Omega_2=\{(0,0)\}$ and $(\Omega_1-(t,0))\cap\Omega_2=\varnothing$ for each $t<0$.
At the same time,
$(\Omega_1+\rho\B-a)\cap(\Omega_2+\rho\B)\ne\varnothing$ for all $a\in\R^2$ and $\rho>0$.
By \cref{prop:geometric_version_extremality}, $\{\Omega_1,\Omega_2\}$ is not
$\mathcal{F}_A$-extremal at $(0,0)$.
\end{example}

\begin{figure}[ht]
    \centering
    \hfill
    \begin{subfigure}[t]{0.4\textwidth}
        \centering
        \includegraphics{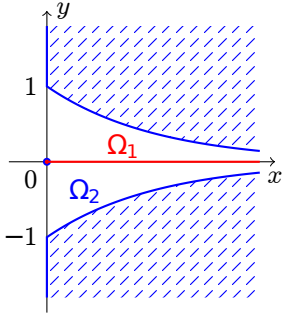}
        \caption{Sets from \cref{ex:extremality_more_restrictive}.}
        \label{fig:counterexample_extremality}
    \end{subfigure}
    \hfill
    \begin{subfigure}[t]{0.4\textwidth}
        \centering
        \includegraphics{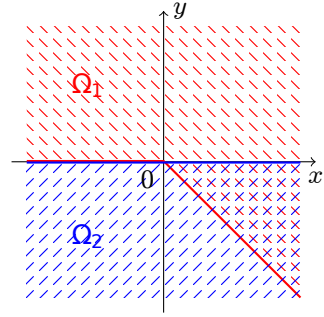}
        \caption{Sets from \cref{ex:generalized_separation}.}
        \label{fig:example_generalized_separation}
    \end{subfigure}
    \hfill
    \caption{Visualization of the sets $\Omega_1$ and $\Omega_2$ from \cref{ex:extremality_more_restrictive,ex:generalized_separation}.}
    \label{fig:sets}
\end{figure}

\Cref{thm:necessary_condition_extremality} produces the following necessary condition for $\mathcal{F}_A$-extremality.

\begin{corollary}\label{cor:NC_extremality}
Assume that
either $X$ is Asplund, or
$\Omega_1,\ldots,\Omega_n$ are convex.
Suppose that
the collection $\{\Omega_1,\ldots,\Omega_n\}$ is
$\mathcal{F}_A$-extremal at $\bx$.
Then,
for each $\eps>0$, there exist points
$x_i\in\Omega_i\cap B_{\eps}(\bx)$ and $x_i^*\in X^*$ $(i=1,\ldots,n)$ satisfying \eqref{eq:NC_extremality_normals_close_to_cone} and
\begin{subequations}\label{eq:NC_extremality_standard}
	\begin{align}
		\label{eq:NC_extremality-1}
			\sum_{i=1}^{n}x_i^*&=0,\\
		\label{eq:NC_extremality-2}
			\sum_{i=1}^{n-1}\norm{x_i^*}&=1.
	\end{align}
\end{subequations}
Moreover, for each $\tau\in(0,1)$, the points
$x_i$ and $x_i^*$ $(i=1,\ldots,n)$ can be chosen so that
\begin{equation}
\label{eq:NC_extremality_additional_estimate}
\sum_{i=1}^{n-1}\ang{x_{i}^*,x_n-x_i+a_i} >\tau\max_{1\le{i}\le{n}-1} \|x_n-x_i+a_i\|,
\end{equation}
where $a_1,\ldots,a_{n-1}$ are vectors satisfying the characterization in \cref{prop:geometric_version_extremality}.
\end{corollary}

\begin{proof}
Fix $\eps>0$ arbitrarily.
Recall that, for each $f_a\in\mathcal{F}_A$, the function
$\hat f_a\colon X^{n}\to\R_\infty$ defined according to
\eqref{eq:hatf} is lower semicontinuous near $\bz:=(\bx,\ldots,\bx)$ relative to $\Omega:=\Omega_1\times\ldots\times\Omega_n$.
By definition of $\mathcal{F}_A$, \cref{prop:geometric_version_extremality}, and \cref{thm:necessary_condition_extremality},
for each $\eta>0$, there exist vectors $a_1,\ldots,a_{n-1}\in X$, points
$x_i\in\Omega_i\cap B_{\eps}(\bx)$, $x'_i\in B_{\eta}(x_i)$, and $x_i^*\in X^*$ $(i=1,\ldots,n)$, and a number $\rho>0$ such that
$\max_{1\le i\le n-1}\|a_i\|<\eps$,
and conditions \eqref{eq:NC_extremality} and \eqref{eq:separation_new} hold, where
$w:=(x'_1-x'_n,\ldots,x'_{n-1}-x'_n,x'_n)$ and
the function $f$ is replaced by $f_a$ defined by \eqref{eq:f_standard_extremality}.
Clearly, we find
\begin{align*}
	\sd f_a(w)
	=
	\sd\|\cdot\|_{X^{n-1}}(x'_1-x'_n-a_1,\ldots,x'_{n-1}-x'_n-a_{n-1}) \times\{0\},
\end{align*}
where $\|\cdot\|_{X^{n-1}}$ is the maximum norm in
$X^{n-1}$.
Condition \eqref{eq:NC_extremality-1}
follows immediately from \eqref{eq:NC_extremality_subdifferential}.
Moreover,
since $f_a(w)>0$, we can apply \cite[Corollary~2.4.16]{Zal02} to find that condition \eqref{eq:NC_extremality-2} is satisfied, and
\begin{equation}\label{eq:NC_extremality_intermediate_finding}
\sum_{i=1}^{n-1}\ang{x_{i}^*,x_n'-x_i'+a_i} = f_a(w).
\end{equation}

Let an arbitrary number
$\tau\in(0,1)$ be fixed, and let $\eta:=\rho(1-\tau)/4$.
In view of \eqref{eq:separation_new}, we have
\begin{equation}\label{eq:NC_extremality_intermediate_finding3}
\max_{1\le{i}\le{n}-1}\|x_n-x_i+a_i\|\ge\rho.
\end{equation}
Using \eqref{eq:f_standard_extremality}, \eqref{eq:NC_extremality-2}, \eqref{eq:NC_extremality_intermediate_finding},
and \eqref{eq:NC_extremality_intermediate_finding3}, we can prove the remaining estimate \eqref{eq:NC_extremality_additional_estimate}:
\begin{align*}
\sum_{i=1}^{n-1}\ang{x_{i}^*,x_n-x_i+a_i}
\ge
&\sum_{i=1}^{n-1}\left(\ang{x_{i}^*,x_n'-x_i'+a_i} -2\norm{x_i^*}\max_{1\le{j}\le{n}}\norm{x_j-x_j'}\right)
\\
>
&\sum_{i=1}^{n-1}\ang{x_{i}^*,x_n'-x_i'+a_i}-2\eta
\\
=
&\max_{1\le{i}\le{n}-1}\|x_n'-x_i'+a_i\|-2\eta
\\
>
&\max_{1\le{i}\le{n}-1}\|x_n-x_i+a_i\|-4\eta
\\
=
&\max_{1\le{i}\le{n}-1}\|x_n-x_i+a_i\|-\rho(1-\tau)
\\
\geq
&\tau\max_{1\le{i}\le{n}-1}\|x_n-x_i+a_i\|.
\end{align*}
This completes the proof.
\end{proof}

The next example illustrates
application of \cref{thm:necessary_condition_extremality} in
the case where $\mathcal F$ consists of discontinuous functions.

\begin{example}
\label{ex:generalized_separation}
Consider the closed sets $\Omega_1,\Omega_2\subset\R^2$ given by
\begin{align*}
	\Omega_1:=\left\{(x,y)\mid\max(y,x+y)\ge0\right\},\quad
	\Omega_2:=\left\{(x,y)\mid y\le0\right\}.
\end{align*}
Let us equip $\R^2$ with the Euclidean norm.
We have $(0,0)\in\Omega_1\cap\Omega_2$ and
$\intr(\Omega_1\cap\Omega_2)=\{(x,y)\,|\,y>0,\, x+y>0\}$.
Hence, these sets cannot be ``pushed apart'', and $\{\Omega_1,\Omega_2\}$ is not
extremal at $(0,0)$ in the conventional sense, see \cref{fig:example_generalized_separation}
for an illustration.
Let the family $\mathcal{F}$ consist of all nonnegative lower semicontinuous functions
$f_t\colon\R^2\times\R^2\to\R_\infty$ of the type
\begin{equation}
\label{eq:non_standard_family_for_extremality}
\forall (x,y),(u,v)\in\R^2\times\R^2\colon\quad
f_t((x,y),(u,v)):=
\norm{(x,y+t)}+i_{(-\infty,0]}(u),
\end{equation}
corresponding to all $t\ge0$.

We now show that $\{\Omega_1,\Omega_2\}$ is $\mathcal{F}$-extremal at $(0,0)$.
Indeed, for each $\eps>0$ and $t\in(0,\eps)$, we have $f_t((0,0),(0,0))=t<\eps$.
The function from \eqref{eq:hatf} takes the form
\[
\forall (x,y),(u,v)\in\R^2\times\R^2\colon\quad
\hat f_t((x,y),(u,v)):=
\norm{(x-u,y-v+t)}+i_{(-\infty,0]}(u).
\]
Let $\rho\in(0,t/3)$, $(x,y)\in\Omega_1+\rho\B$, and $(u,v)\in\Omega_2+\rho\B$.
If $u>0$ or $x\ne u$, then $\hat f_t((x,y),(u,v))>0$.
Let $x=u\le0$.
Then $y>-2\rho$, $v<\rho$, and, consequently, $\hat f_t((x,y),(u,v))=|y-v+t|>-3\rho+t>0$.
Hence, condition \eqref{eq:extremality_nonnegativity} holds, i.e., $\{\Omega_1,\Omega_2\}$ is $\mathcal{F}$-extremal at $(0,0)$.

For each $t\ge0$, $\hat f_t$ is Lipschitz continuous on $\dom\hat f_t=\R^2\times((-\infty,0]\times\R)$
and, for every point $((x,y),(u,v))\in\dom\hat f_t$, the distance $\dist_{\Omega_1\times\Omega_2}((x,y),(u,v))$ is attained
at some point $((x',y'),(u',v'))$ with $u'=u$, i.e., $((x',y'),(u',v'))\in\dom\hat f_t$.
Using this, it is easy to see from \cref{def:lower_semicontinuity_relative_to_set}\,\ref{item:def_lower_semicontinuity_relative_to_set_near_point}
that $\hat f_t$ is lower semicontinuous near $((0,0),(0,0))$ relative to $\Omega_1\times\Omega_2$.

By \cref{thm:necessary_condition_extremality},
for each $\eps>0$, there exist a number $t\in(0,\eps)$ and points
$(x,y)\in\Omega_1\cap B_{\eps}(0,0)$, $(u,v)\in\Omega_2\cap B_{\eps}(0,0)$, $(x^*,y^*),(u^*,v^*)\in\R^2$, and
$w\in X^2\times X^2$
such that
$0<f_t(w)<\infty$ and
\begin{subequations}\label{eq:non_trivial_generalized_separation}
	\begin{align}
	\label{eq:non_trivial_generalized_separation_distance_to_cone}
	\dist_{N_{\Omega_1}(x,y)}\left((x^*,y^*)\right)+ \dist_{{N}_{\Omega_2}(u,v)}\left((u^*,v^*)\right)<\eps,&
	\\
	\label{eq:non_trivial_generalized_separation_subdifferential}
	-\left((x^*,y^*),(x^*,y^*)+(u^*,v^*)\right)\in\sd f_t(w).
	\end{align}
\end{subequations}
In view of \eqref{eq:non_standard_family_for_extremality},
it follows from \eqref{eq:non_trivial_generalized_separation_subdifferential} that $\norm{(x^*,y^*)}=1$, $x^*+u^*\leq 0$, and $y^*+v^*=0$.
When $\eps$ is sufficiently small, condition \eqref{eq:non_trivial_generalized_separation_distance_to_cone} implies one of the following situations:
\begin{itemize}
	\item $x<0$, $y=v=0$, and $(x^*,y^*)$ as well as $(u^*,v^*)$ can be made arbitrarily close to
		$(0,-1)$ and $(0,1)$, respectively,
	\item $x>0$, $y=-x$, $v=0$, and $(x^*,y^*)$ as well as $(u^*,v^*)$ can be made
		arbitrarily close to $(-\sqrt 2/2,-\sqrt 2/2)$ and $(0,\sqrt 2/2)$, respectively.
\end{itemize}
This can be interpreted as a kind of generalized separation.
\end{example}

\section{Geometrically-constrained optimization problems with composite objective function}
\label{sec:geometric_constraints}

In this section, we are going to apply the theory of \cref{sec:main_result} to the optimization problem
\begin{equation}\label{eq:non_Lipschitz_objective}\tag{Q}
	\min\{f(x)+q(x)\,|\,G(x)\in K,\,x\in C\}
\end{equation}
where $f\colon X\to\R$ is continuously Fr\'{e}chet differentiable, $q\colon X\to\R_\infty$ is 
lower semicontinuous, $G\colon X\to Y$ is continuously Fr\'{e}chet differentiable, and
$C\subset X$ as well as $K\subset Y$ are nonempty and closed.
Here, $X$ and $Y$ are assumed to be Banach spaces.
Throughout the section, the feasible set of \eqref{eq:non_Lipschitz_objective} will be denoted
by $\mathcal S$, and we implicitly assume $\mathcal S\cap\dom q\neq\varnothing$ in order
to avoid trivial situations.

Observe that the objective function $\varphi:=f+q$ can be decomposed into a regular part $f$ and
some challenging irregular part $q$ while the
constraints in \eqref{eq:non_Lipschitz_objective} are stated in so-called geometric form.
In this regard, the model \eqref{eq:non_Lipschitz_objective} still covers numerous applications
ranging from data science and image processing (in case where $q$ is a sparsity-promoting functional)
over conic programs (in which case $K$ is a convex cone) to disjunctive programs
which comprise, exemplary, complementarity- and cardinality-constrained problems (in this situation,
$K$ is a nonconvex set of combinatorial structure).

In the subsequently stated remark, we embed program \eqref{eq:non_Lipschitz_objective} into the
rather general framework which has been discussed in \cref{sec:main_result}.
\begin{remark}\label{rem:embedding_Q}
	Observing that $f$ is differentiable, we find
	\[
		\forall x\in X\colon\quad
		\sd \varphi(x)=\sd(f+q)(x)=f'(x)+\sd q(x)
	\]
	from the sum rule stated in \cite[Corollary~1.12.2]{Kru03}.
	The feasibility mapping $\Phi\colon X\tto Y\times X$ associated with
	\eqref{eq:non_Lipschitz_objective} is given by means of
	$\Phi(x):=(G(x)-K,x-C)$ for all $x\in X$, see \cref{ex:upper_semicontinuity}.
	We find
	\begin{equation}\label{eq:gph_Phi}
		\gph\Phi
		=
		\{(x,(y,x'))\in X\times Y\times X\,|\,(G(x)-y,x-x')\in K\times C\}.
	\end{equation}
	Observing that the continuously differentiable mapping $(x,y,x')\mapsto(G(x)-y,x-x')$
	possesses a surjective derivative,
	we can apply the change-of-coordinates formula from
	\cite[Corollary~1.15]{Mordukhovich2006} in order to obtain
	\[
		N_{\gph\Phi}(x,(y,x'))
		=
		\left\{(G'(x)^*\lambda+\eta,-\lambda,-\eta)\in X^*\times Y^*\times X^*\,\middle|\,
			\begin{aligned}
				&\lambda\in N_K(G(x)-y),\\
				&\eta\in N_C(x-x')
			\end{aligned}
			\right\}
	\]
	for each triplet $(x,(y,x'))\in\gph\Phi$, and this yields
	\[
		D^*\Phi(x,(y,x'))(\lambda,\eta)
		=
		\begin{cases}
			G'(x)^*\lambda+\eta	&	\text{if }\lambda\in N_K(G(x)-y),\,\eta\in N_C(x-x'),\\
			\varnothing			&	\text{otherwise}
		\end{cases}
	\]
	for arbitrary $\lambda\in Y^*$ and $\eta\in X^*$.
\end{remark}

\subsection{Approximate stationarity and uniform qualification condition}\label{sec:asymptotic_stuff}

The subsequent theorem is a simple consequence of \cref{cor:necessary_conditions_local_minimizer}
and \cref{rem:embedding_Q},
and provides a necessary optimality condition for \eqref{eq:non_Lipschitz_objective}.

\begin{theorem}\label{thm:asymptotic_stationarity_necessary_optimality_condition}
	Fix $\bar x\in\mathcal S\cap\dom q$ and assume that
	\begin{enumerate}
		\item \label{item:lower_semicontinuity_assumption_geometric_constraints}
			the function $f+q$ is lower semicontinuous near $\bar x$ relative to $\Phi$ from \cref{rem:embedding_Q} at $(0,0)$;
		\item \label{item:asplund_or_convex_geometric_constraints}
			either $X$ and $Y$ are Asplund, or $f$, $q$, and $\gph\Phi$ from \eqref{eq:gph_Phi} are convex.
	\end{enumerate}
	Suppose that $\bx$ is a local minimizer of \eqref{eq:non_Lipschitz_objective}.
	Then, for each $\eps>0$, there exist points $x,x',x''\in B_\eps(\bx)$ and $y\in\eps\B$ such that $|q(x)-q(\bar x)|<\varepsilon$ and
	\begin{equation}\label{eq:approximate_stationarity}
		0\in f'(x)+\sd q(x)+G'(x')^*N_K(G(x')-y)+N_C(x'')+\eps\B^*.
	\end{equation}
\end{theorem}

In the subsequent remark, we comment on some special situations where the assumptions of \cref{thm:asymptotic_stationarity_necessary_optimality_condition}
are naturally valid and which can be checked in terms of initial data.
\begin{remark}\label{rem:exemplary_settings}
	Let $\bar x\in\mathcal S\cap\dom q$. Due to
\cref{prop:decomposition_lsc,cor:lower_semicontinuity_svm_via_weak_sequential_lower_semicontinuity,ex:upper_semicontinuity,cor:lower_semicontinuity_relative_to_svm_decomposition},
	each of the following conditions implies condition~\ref{item:lower_semicontinuity_assumption_geometric_constraints} of
	\cref{thm:asymptotic_stationarity_necessary_optimality_condition}:
	\begin{enumerate}
		\item \label{item:setting_Asplund}
			the function $f+q$ satisfies one of the conditions~\ref{item:lsc_via_interiority}-\ref{item:lsc_via_uniform_continuity}
			in \cref{cor:sufficient_criteria_lower_semicontinuity_relative_to_a_set} and
			the mapping $\Phi$ from \cref{rem:embedding_Q} is metrically subregular at $(\bx,(0,0))$,
			see \cref{ex:upper_semicontinuity};
		\item \label{item:setting_reflexive}
			$X$ is reflexive, the functions $f$ and $q$ are weakly sequentially lower semicontinuous,
			and condition \eqref{eq:partially_weakly_sequantially_usc_geonetric_constraints}
			holds for all sequences $\{x_k\}_{k\in\N}\subset X$ and all points $x\in X$.
	\end{enumerate}
	Furthermore, condition~\ref{item:asplund_or_convex_geometric_constraints} of
	\cref{thm:asymptotic_stationarity_necessary_optimality_condition} is valid whenever $X$ and $Y$ are Asplund, or if
	$f$, $q$, and $C$ are convex, $K$ is a convex cone, and $G$ is $K$-convex in the following sense:
	\[
		\forall x,x'\in X\,\forall s\in[0,1]\colon\quad
		G(sx+(1-s)x')-s\,G(x)-(1-s)G(x')\in K.
	\]
\end{remark}

We note that \eqref{eq:non_Lipschitz_objective} already satisfies condition~\ref{item:setting_reflexive} of \cref{rem:exemplary_settings}
as soon as $X$ and $Y$ are finite-dimensional.
In the presence of condition~\ref{item:setting_reflexive} from \cref{rem:exemplary_settings},
\cref{thm:asymptotic_stationarity_necessary_optimality_condition} is closely related to
\cite[Proposition~3.3]{BoergensKanzowMehlitzWachsmuth2019} as soon as $q$ is absent.

Due to \cref{thm:asymptotic_stationarity_necessary_optimality_condition}, the following definition is reasonable.
\begin{definition}\label{def:asymptotic_stationarity}
	A point $\bar x\in \mathcal S\cap\dom q$ is an \emph{approximately stationary point} of \eqref{eq:non_Lipschitz_objective}
	if, for each $\varepsilon>0$, there exist points $x,x',x''\in B_\eps(\bx)$ and $y\in\eps\B$ such that $|q(x)-q(\bar x)|<\varepsilon$
	and \eqref{eq:approximate_stationarity} are valid.
\end{definition}

Approximate necessary optimality conditions in terms of Fr\'{e}chet subgradients and normals can be traced back
to the 1980s, see e.g.\ \cite{KruMor80,Kruger1985} and the references therein.

In order to compare the notion of stationarity from \cref{def:asymptotic_stationarity} to others from the literature, let us mention
an equivalent characterization of asymptotic stationarity in terms of sequences.
\begin{remark}\label{rem:asymptotic_stationarity_via_sequences}
	A point $\bar x\in\mathcal S\cap\dom q$ is approximately stationary if and only if
	there are sequences
	$\{x_k\}_{k\in\N},\{x_k'\}_{k\in\N},\{x_k''\}_{k\in\N}\subset X$, $\{y_k\}_{k\in\N}\subset Y$,
	and $\{\eta_k\}_{k\in\N}\subset X^*$ such that $x_k\to\bar x$, $x_k'\to\bar x$, $x_k''\to\bar x$,
	$y_k\to 0$, $\eta_k\to 0$, $q(x_k)\to q(\bar x)$, and
	\[
		\forall k\in\N\colon\quad
		\eta_k\in f'(x_k)+\sd q(x_k)+G'(x_k')^*N_K(G(x_k')-y_k)+N_C(x_k'').
	\]
\end{remark}

In case where $X$ and $Y$ are finite-dimensional while $q$ is locally Lipschitzian, a similar
approximate stationarity condition in terms of sequences has been investigated in \cite[Sections~4, 5.1]{Mehlitz2020b}. In
\cite{BoergensKanzowMehlitzWachsmuth2019}, the authors considered the model
\eqref{eq:non_Lipschitz_objective} with convex sets $K$ and $C$ in the absence of $q$.
Generally, using approximate notions of stationarity in nonlinear programming can be traced back
to \cite{AndreaniMartinezSvaiter2010,AndreaniHaeserMartinez2011}.
Let us mention that in all these papers, the authors speak of \emph{asymptotic} or \emph{sequential} stationarity
conditions.
A sequential Lagrange multiplier rule for convex programs in Banach spaces can be found already in \cite{Thibault1997}.

During the last decade, the concept of approximate stationarity
has been extended to several classes of optimization problems comprising, exemplary,
complementarity- and cardinality-constrained programs,
see \cite{AndreaniHaeserSecchinSilva2019,KanzowRaharjaSchwartz2021,Ramos2019},
conic optimization problems,
see \cite{AndreaniHaeserViana2020},
smooth geometri\-cally-constrained optimization problems in Banach spaces,
see \cite{BoergensKanzowMehlitzWachsmuth2019},
and nonsmooth Lipschitzian optimization problems in finite-dimensional spaces,
see \cite{Mehlitz2020b,Mehlitz2021}.
In each of the aforementioned situations, it has been demonstrated that approximate stationarity,
on the one hand, provides a necessary optimality condition in the absence of constraint qualifications, and
\cref{thm:asymptotic_stationarity_necessary_optimality_condition} demonstrates that this is the case for
our concept from \cref{def:asymptotic_stationarity} as well under reasonable assumptions.
On the other hand, the results from the literature underline that approximate stationarity is naturally satisfied
for accumulation points of sequences generated by some solution algorithms.
In \cref{sec:alm}, we extend these observations to the present setting.

Assume that $\bar x\in \mathcal S\cap\dom q$ is an approximately stationary point of \eqref{eq:non_Lipschitz_objective}.
Due to \cref{rem:asymptotic_stationarity_via_sequences}, we find sequences $\{x_k\}_{k\in\N},\{x_k'\}_{k\in\N},\{x_k''\}_{k\in\N}\subset X$,
$\{y_k\}_{k\in\N}\subset Y$,
and $\{\eta_k\}_{k\in\N}\subset X^*$ satisfying $x_k\to\bar x$, $x_k'\to\bar x$, $x_k''\to\bar x$, $y_k\to 0$, $\eta_k\to 0$, $q(x_k)\to q(\bar x)$, and
$\eta_k\in f'(x_k)+\partial q(x_k)+G'(x_k')^*N_K(G(x_k')-y_k)+N_C(x_k'')$ for each $k\in\N$.
Particularly, we find sequences $\{\lambda_k\}_{k\in\N}\subset Y^*$ and $\{\mu_k\}_{k\in\N}\subset X^*$
of multipliers and a sequences $\{\xi_k\}_{k\in\N}\subset X^*$ of subgradients such that
$\eta_k=f'(x_k)+\xi_k+G'(x_k')^*\lambda_k+\mu_k$, $\lambda_k\in N_K(G(x_k')-y_k)$,
$\mu_k\in N_C(x_k'')$, and $\xi_k\in\partial q(x_k)$ for each $k\in\N$. Assuming for a moment
$\lambda_k\weaklystar\lambda$, $\mu_k\weaklystar\mu$, and $\xi_k\weaklystar\xi$ for some
$\lambda\in Y^*$ and $\mu,\xi\in X^*$, we find $\lambda\in\overline N_K(G(\bar x))$,
$\mu\in\overline N_C(\bar x)$, and $\xi\in\bsd q(\bar x)$ by definition of the limiting normal cone
and subdifferential, respectively, as well as $0=f'(\bar x)+\xi+G'(\bar x)^*\lambda+\mu$,
i.e., a multiplier rule is valid at $\bar x$ which is referred to as M-stationarity in the literature.

\begin{definition}\label{def:M_stationarity}
	A feasible point $\bar x\in\mathcal S\cap\dom q$ is an \emph{M-stationary} point of \eqref{eq:non_Lipschitz_objective}
	if
	\[
		0\in f'(\bar x)+\bsd q(\bar x)+G'(\bar x)^*\overline N_K(G(\bar x))+\overline N_C(\bar x).
	\]
\end{definition}

Let us note that in the case of standard nonlinear programming, where $q$ vanishes while $C:=X$, $Y:=\R^{m_1+m_2}$, and $K:=(-\infty,0]^{m_1}\times\{0\}^{m_2}$
for $m_1,m_2\in\N$,
the system of M-stationarity coincides with the classical Karush--Kuhn--Tucker system.

One can easily check by means of simple examples that approximately stationary points of \eqref{eq:non_Lipschitz_objective}
do not need to be M-stationary even in finite dimensions. Roughly speaking, this phenomenon is caused by the fact that
the multiplier and subgradient sequences $\{\lambda_k\}_{k\in\N}$, $\{\mu_k\}_{k\in\N}$, and $\{\xi_k\}_{k\in\N}$
in the considerations which prefixed \cref{def:M_stationarity} do not
need to be bounded, see \cite[Section~3.1]{Mehlitz2020b} for related observations.
The following example is inspired by \cite[Example~3.3]{Mehlitz2020b}.

\begin{example}\label{ex:approximate_vs_M_stationarity}
	We consider $X=Y=C:=\R$, set $f(x):=x$, $q(x):=0$, as well as $G(x):=x^2$ for all $x\in\R$, and fix $K:=(-\infty,0]$.
	Let us investigate $\bar x:=0$.
	Note that this is the only feasible point of the associated optimization problem \eqref{eq:non_Lipschitz_objective}
	and, thus, its uniquely determined global minimizer.	
	Due to $f'(\bar x)=1$ and $G'(\bar x)=0$, $\bar x$ cannot be an M-stationary point
	of \eqref{eq:non_Lipschitz_objective}.
	On the other hand, setting
	\[
		x_k:=0,\quad x_k':=-\frac{1}{2k},\quad y_k:=\frac{1}{4k^2},\quad\eta_k:=0,\quad\lambda_k:=k
	\]
	for each $k\in\N$, we have $x_k\to\bar x$, $x_k'\to\bar x$, $y_k\to 0$, $\eta_k\to 0$, as well as
	$\eta_k=f'(x_k)+G'(x_k')^*\lambda_k$ and $\lambda_k\in N_K(G(x_k')-y_k)$ for each $k\in\N$, i.e.,
	$\bar x$ is approximately stationary for \eqref{eq:non_Lipschitz_objective}.
	Observe that $\{\lambda_k\}_{k\in\N}$ is unbounded.
\end{example}

Let us underline that the above example demonstrates that local minimizers of \eqref{eq:non_Lipschitz_objective}
do not need to be M-stationary in general while approximate stationarity serves as a necessary optimality
condition under some assumptions on the data which are inherent in finite dimensions,
see \cref{thm:asymptotic_stationarity_necessary_optimality_condition} and \cref{rem:exemplary_settings}.
Nevertheless, M-stationarity turned out to be a celebrated stationarity condition in finite-dimensional
optimization. On the one hand, it is restrictive enough to exclude non-reasonable feasible points of
\eqref{eq:non_Lipschitz_objective} when used as a necessary optimality condition. On the other hand,
it is weak enough to hold at the local minimizers of \eqref{eq:non_Lipschitz_objective} under very
mild qualification conditions.
Exemplary, we would like to refer the reader to \cite{FlegelKanzowOutrata2007} where this is visualized
by so-called disjunctive programs where $K$ is the union of finitely many polyhedral sets.
Another interest in M-stationarity arises from the fact that this system can often be solved directly
in order to identify reasonable feasible points of \eqref{eq:non_Lipschitz_objective}, see e.g.\
\cite{GuoLinYe2015,HarderMehlitzWachsmuth2021}.
In infinite-dimensional optimization, particularly, in optimal control, M-stationarity has turned out
to be of limited practical use since the limiting normal cone to nonconvex sets in function spaces
is uncomfortably large due to convexification effects arising when taking weak limits, see e.g.\
\cite{HarderWachsmuth2018,MehlitzWachsmuth2018}.

Due to this interest in M-stationarity, at least from the finite-dimensional point of view,
we aim to find conditions guaranteeing that a given approximately stationary point of \eqref{eq:non_Lipschitz_objective} is
already M-stationary.

\begin{definition}\label{def:asymptotic_regularity}
	We say that the \emph{uniform qualification condition} holds at $\bar x\in\mathcal S\cap\dom q$ whenever
	\begin{align*}
		\limsup\limits_{\substack{x\to\bar x,\,x'\to\bar x,\,x''\to\bar x,\\ y\to 0,\,q(x)\to q(\bar x)}}
		&\left(\partial q(x)+G'(x')^*N_K(G(x')-y)+N_C(x'')\right)
		\\
		&
		\subset
		\bsd q(\bar x)+G'(\bar x)^*\overline N_K(G(\bar x))+\overline N_C(\bar x).
	\end{align*}
\end{definition}

By construction, the uniform qualification condition guarantees that a given approximately stationary
point of \eqref{eq:non_Lipschitz_objective} is already M-stationary as desired.

\begin{proposition}\label{prop:from_approximate_to_M_stationarity}
Let $\bar x\in\mathcal S\cap\dom q$
satisfy the uniform qualification condition.
If $\bar x$ is
an approximately stationary
point of \eqref{eq:non_Lipschitz_objective},
then it is M-stationary.
\end{proposition}
\begin{proof}
	By definition of approximate stationarity, for each $k\in\N$, we find $x_k,x'_k,x''_k\in B_{1/k}(\bar x)$,
	$y_k\in \tfrac1k\mathbb B$, and $\eta_k\in\tfrac1k\mathbb B^*$ such that $|q(x_k)-q(\bar x)|<\tfrac1k$ and
	$\eta_k-f'(x_k)\in \partial q(x_k)+G'(x'_k)^*N_K(G(x'_k)-y_k)+N_C(x''_k)$.
	Since $f$ is assumed to be continuously differentiable, we find $\eta_k-f'(x_k)\to -f'(\bar x)$.
	Thus, by validity of the uniform qualification condition, it holds
	\begin{align*}
		-f'(\bar x)
		&\in
		\limsup\limits_{k\to+\infty}\left(\partial q(x_k)+G'(x'_k)^*N_K(G(x'_k)-y_k)+N_C(x''_k)\right)\\
		&
		\subset
		\bsd q(\bar x)+G'(\bar x)^*\overline N_K(G(\bar x))+\overline{N}_C(\bar x),
	\end{align*}
	i.e., $\bar x$ is an M-stationary point of \eqref{eq:non_Lipschitz_objective}.
\end{proof}

Combining this with \cref{thm:asymptotic_stationarity_necessary_optimality_condition} yields the following result.

\begin{corollary}\label{cor:asymptotic_regularity_CQ}
Let $\bar x\in\mathcal S\cap\dom q$ be a local minimizer of \eqref{eq:non_Lipschitz_objective} which
satisfies the assumptions of \cref{thm:asymptotic_stationarity_necessary_optimality_condition}
as well as the uniform qualification condition. Then $\bar x$ is M-stationary.
\end{corollary}

Observe that we do not need any so-called \emph{sequential normal compactness condition}, see
\cite[Section~1.1.4]{Mordukhovich2006}, for the above statement to hold which pretty much contrasts
the results obtained in \cite[Section~5]{Mordukhovich2006}. Indeed, sequential normal compactness
is likely to fail in the function space context related to optimal control, see \cite{Mehlitz2019a}.

Let us point the reader's attention to the fact that the uniform qualification condition is not a constraint
qualification in the narrower sense for \eqref{eq:non_Lipschitz_objective} since it also depends
on (parts of) the objective function.
Nevertheless, \cref{cor:asymptotic_regularity_CQ} shows that it may serve as a qualification
condition for M-stationarity of local minimizers under mild assumptions on the data.
In the absence of $q$, the uniform qualification condition is related
to other prominent so-called \emph{sequential} or \emph{asymptotic} constraint qualifications from the literature which address
several different kinds of optimization problems, see e.g.
\cite{AndreaniFazzioSchuverdtSecchin2019,AndreaniHaeserSecchinSilva2019,AndreaniMartinezRamosSilva2016,
BoergensKanzowMehlitzWachsmuth2019,Mehlitz2020b,Mehlitz2021,Ramos2019}.
In \cref{sec:control}, we demonstrate by means of a prominent setting from optimal control that
the uniform qualification condition may hold in certain situations where $q$ is present,
see \cref{lem:asymptotic_regularity_OC}.

\begin{remark}\label{rem:uniform_CQ}
	Note that in the particular setting $q\equiv 0$, the uniform qualification condition from \cref{def:asymptotic_regularity}
	at some point $\bar x\in\mathcal S$ simplifies to
	\begin{equation}\label{eq:uniform_CQ}
		\limsup\limits_{x'\to\bar x,\,x''\to \bar x,\,y\to 0}
		 \bigl(G'(x')^*N_K(G(x')-y)+N_C(x'')\bigr)
		\subset G'(\bar x)^*\overline N_K(G(\bar x))+\overline{N}_C(\bar x).
	\end{equation}
	In the light of \cref{prop:from_approximate_to_M_stationarity} and \cref{cor:asymptotic_regularity_CQ},
	\eqref{eq:uniform_CQ} serves as a constraint qualification guaranteeing M-stationarity of $\bar x$
	under mild assumptions
	as soon as this point is a local minimizer of the associated problem \eqref{eq:non_Lipschitz_objective}.
	One may, thus, refer to \eqref{eq:uniform_CQ} as the \emph{uniform constraint qualification}.
\end{remark}

Observations related to the ones from \cref{rem:uniform_CQ} have been made in \cite{BoergensKanzowMehlitzWachsmuth2019},
\cite[Section~2.2]{JiaKanzowMehlitzWachsmuth2021}, and \cite[Section~5.1]{Mehlitz2020b}
and underline that \eqref{eq:uniform_CQ} is a comparatively weak constraint qualification whenever $q\equiv 0$.
Exemplary, let us mention that whenever $X$ and $Y$ are finite-dimensional the \emph{generalized Mangasarian--Fromovitz
constraint qualification}
\begin{equation}\label{eq:GMFCQ}
	-G'(\bar x)^*\lambda\in\overline N_C(\bar x),\,\lambda\in\overline N_K(G(\bar x))
	\quad\Longrightarrow\quad
	\lambda=0
\end{equation}
is sufficient for \eqref{eq:uniform_CQ} to hold, but the uniform constraint qualification
is often much weaker than \eqref{eq:GMFCQ} which corresponds to metric regularity of $\Phi$
from \cref{rem:embedding_Q} at $(\bar x,(0,0))$, see \cite[Section~3.2]{Mehlitz2020b} for related discussions.
Let us also mention that \eqref{eq:GMFCQ} is sufficient for metric subregularity of $\Phi$ at $(\bar x,(0,0))$ exploited in
\cref{cor:lower_semicontinuity_relative_to_svm_decomposition}.

The following proposition provides a sufficient condition for validity of the uniform qualification condition in
case where $X$ is finite-dimensional.

\begin{proposition}\label{lem:sufficient_condition_sequential_regularity}
Let $X$ be finite-dimensional and $\bar x\in \mathcal S\cap\dom q$.
Suppose that the uniform constraint qualification \eqref{eq:uniform_CQ} is valid at $\bar x$, and
\begin{equation}
\label{eq:BCQ}
	\bigl(G'(\bar x)^*\overline N_K(G(\bar x))+\overline N_C(\bar x)\bigr)
		\cap(-\bsd^\infty q(\bar x))=\{0\}.
\end{equation}
	Then the uniform qualification condition holds at $\bar x$.
\end{proposition}
\begin{proof}
	Let us fix
	\[
		x^*\in\limsup\limits_{\substack{x\to\bar x,\,x'\to\bar x,\,x''\to\bar x,\\ y\to 0,\ q(x)\to q(\bar x)}}
		\left(\partial q(x)+G'(x')^*N_K(G(x')-y)+N_C(x'')\right).
	\]
	Then we find sequences $\{x_k\}_{k\in\N},\{x_k'\}_{k\in\N},\{x_k''\}_{k\in\N}\subset X$,
	$\{y_k\}_{k\in\N}\subset Y$,
	and $\{x_k^*\}_{k\in\N}\subset X^*$ such that $x_k\to\bar x$, $x_k'\to\bar x$, $x_k''\to\bar x$,
	$y_k\to\bar y$, $q(x_k)\to q(\bar x)$, and $x_k^*\to x^*$ as well as
	$x_k^*\in\partial q(x_k)+G'(x'_k)^*N_K(G(x'_k)-y_k)+N_C(x''_k)$ for all $k\in\N$.
	Thus, there are sequences $\{u_k^*\}_{k\in\N},\{v^*_k\}_{k\in\N}\subset X^*$
	satisfying
	$x_k^*=u_k^*+v_k^*$, $u_k^*\in\partial q(x_k)$, and
	$v_k^*\in G'(x_k')^*N_K(G(x_k')-y_k)+ N_C(x_k'')$ for all $k\in\N$.
	
	Let us assume that $\{u_k^*\}_{k\in\N}$ is unbounded. Then, due to $x_k^*\to x^*$,
	$\{v_k^*\}_{k\in\N}$ is unbounded, too.
	For each $k\in\N$, we define
	$\tilde u_k^*:=u_k^*/(\norm{u_k^*}+\norm{v_k^*})$ and
	$\tilde v_k^*:=v_k^*/(\norm{u_k^*}+\norm{v_k^*})$, i.e.,
	the sequence $\{(\tilde u_k^*,\tilde v_k^*)\}_{k\in\N}$ belongs to the unit sphere of
	$X^*\times X^*$. Without loss of generality, we may assume $\tilde u_k^*\to\tilde u^*$ and
	$\tilde v_k^*\to\tilde v^*$ for some $\tilde u^*,\tilde v^*\in X^*$ since $X$
	is finite-dimensional. We note that $\tilde u^*$ and $\tilde v^*$ cannot vanish at the same time.
	Taking the limit in $x_k^*/(\norm{u_k^*}+\norm{v_k^*})=\tilde u_k^*+\tilde v_k^*$,
	we obtain $0=\tilde u^*+\tilde v^*$.
	By definition of the singular limiting subdifferential, we have $\tilde u^*\in\bsd^\infty q(\bar x)$
	while
	\[
		\tilde v^*
		\in
		\limsup
		\limits_{k\to+\infty}\bigl(G'(x_k')^*N_K(G(x_k')-y_k)+N_C(x_k'')\bigr)
		\subset
		G'(\bar x)^*\overline N_K(G(\bar x))+\overline N_C(\bar x)
	\]
	follows by the uniform constraint qualification \eqref{eq:uniform_CQ}.
	Thus, we find $\tilde u^*=\tilde v^*=0$ from condition \eqref{eq:BCQ}.
	The latter, however, contradicts $(\tilde u^*,\tilde v^*)\neq(0,0)$.
	
	From above, we now know that $\{u_k^*\}_{k\in\N}$ and $\{v_k^*\}_{k\in\N}$ are bounded. Without
	loss of generality, we may assume $u_k^*\to u^*$ and $v_k^*\to v^*$ for some $u^*,v^*\in X^*$.
	By definition of the limiting subdifferential we have $u^*\in\bsd q(\bar x)$, and
	$v^*\in G'(\bar x)^*\overline N_K(G(\bar x))+\overline N_C(\bar x)$
	is guaranteed by the uniform constraint qualification \eqref{eq:uniform_CQ}.
	Thus, we end up with
	$x^*\in \bsd q(\bar x)+G'(\bar x)^*\overline N_K(G(\bar x))+\overline N_C(\bar x)$
	which completes the proof.
\end{proof}

\Cref{lem:sufficient_condition_sequential_regularity} shows that in case where $X$ is finite-dimensional,
validity of the uniform qualification condition can be guaranteed in the presence of two conditions.
The first one, represented by condition \eqref{eq:uniform_CQ},
is a sequential constraint
qualification which guarantees regularity of the constraints at the reference point.
The second one, given by condition \eqref{eq:BCQ}, ensures in some sense that the challenging
part of the objective function and the constraints of \eqref{eq:non_Lipschitz_objective} are somewhat
compatible at the reference point. A similar \emph{decomposition} of qualification conditions has been
used in \cite{ChenGuoLuYe2017,GuoYe2018} in order to ensure M-stationarity of standard nonlinear problems in
finite dimensions with a composite objective function. In the latter papers, the authors referred
to a condition of type \eqref{eq:BCQ} as \emph{basic qualification}, and this terminology can be
traced back to the works of Mordukhovich, see e.g.\ \cite{Mordukhovich2006}.

Note that in order to transfer \cref{lem:sufficient_condition_sequential_regularity} to the
infinite-dimensional setting, one would be in need to postulate sequential compactness
properties on $q$ or the constraint data which are likely to fail in several interesting function
spaces, see \cite{Mehlitz2019a} again.

\subsection{Augmented Lagrangian methods for optimization problems with non-Lipschitzian objective functions}\label{sec:alm}

We consider the optimization problem \eqref{eq:non_Lipschitz_objective}
such that $X$ is an Asplund space, $Y$ is a Hilbert space with $Y\cong Y^*$,
and $K$ is convex.
Let us note that the assumption on $Y$ can be relaxed by assuming the existence of a Hilbert space
$H$ with $H\cong H^*$ such that $(Y,H,Y^*)$ is a Gelfand triplet, see
\cite[Section~7]{BoergensKanzowMehlitzWachsmuth2019} or \cite{BoergensKanzowSteck2019,KanzowSteckWachsmuth2018} for a discussion.
Furthermore, we will exploit the following assumption which is standing throughout this section.
\begin{assumption}\label{ass:ALM}
	At least one of the following assumptions is valid.
	\begin{enumerate}
		\item The space $X$ is finite-dimensional.
		\item The function $q$ is uniformly continuous.
		\item The functions $f$, $q$, and $x\mapsto \dist_K^2(G(x))$ are weakly sequentially lower semicontinuous
			and $C$ is weakly sequentially closed. Furthermore, $X$ is reflexive.
	\end{enumerate}
\end{assumption}

Throughout this
subsection, we assume that $C$ is a comparatively simple set, e.g., a box if $X$ is equipped with
a (partial) order relation, while the constraints $G(x)\in K$ are difficult and will be treated
with the aid of a multiplier-penalty approach.
In this regard, for some penalty parameter
$\theta>0$, we investigate the (partial) augmented Lagrangian function
$\LL_\theta\colon X\times Y\to\R_\infty$ given by
\[
	\forall (x,\lambda)\in X\times Y\colon\quad
	\LL_\theta(x,\lambda)
	:=
	f(x)
	+\frac{\theta}{2}\dist_K^2\left(G(x)+\frac{\lambda}{\theta}\right)
	+q(x).
\]
We would like to point the reader's attention to the fact that the second
summand in the definition of $\LL_\theta$ is
continuously differentiable since the squared distance to a convex set possesses
this property.
For the control of the penalty parameter, we make use of the function $V_\theta\colon X\times Y\to\R$
given by
\[
	\forall (x,y)\in X\times Y\colon\quad
	V_\theta(x,\lambda):=\norm{G(x)-P_K(G(x)+\lambda/\theta)}.
\]
The method of interest is now given as stated in \cref{alg:ALM}.
\begin{algorithm}
	\begin{enumerate}[label=(S.\arabic*)]
		\setcounter{enumi}{-1}
		\item Choose $(x_0,\lambda_0)\in (\dom q)\times Y$, $\theta_0>0$, $\gamma>1$,
			$\tau\in(0,1)$, and a nonempty, bounded set $B\subset Y$ arbitrarily.
			Set $k:=0$.
		\item \label{item:termination_ALM}
			If $(x_k,\lambda_k)$ satisfies a suitable termination criterion,
			then stop.
		\item \label{item:subproblem_ALM}
			Choose $u_k\in B$ and find an approximate solution $x_{k+1}\in C\cap\dom q$ of
			\begin{equation}\label{eq:ALM_subproblem}
				\min\{\LL_{\theta_{k}}(x,u_k)\,|\,x\in C\}.
			\end{equation}			
		\item \label{item:update_of_multipliers}
			Set
			\[
				\lambda_{k+1}:=\theta_k\left[G(x_{k+1})+u_k/\theta_k
					-P_K\left(G(x_{k+1})+u_k/\theta_k\right)\right].
			\]
		\item \label{item:control_of_parameter_ALM}
			If $k=0$ or $V_{\theta_k}(x_{k+1},u_k)\leq\tau\,V_{\theta_{k-1}}(x_k,u_{k-1})$,
			then set $\theta_{k+1}:=\theta_k$. Otherwise, set $\theta_{k+1}:=\gamma\,\theta_k$.
		\item Go to \ref{item:termination_ALM}.
	\end{enumerate}
	\caption{
		Safeguarded augmented Lagrangian method for \eqref{eq:non_Lipschitz_objective}.
		\label{alg:ALM}
		}
\end{algorithm}

We would like to point the reader's attention to the fact that \cref{alg:ALM} is a so-called
\emph{safeguarded} augmented Lagrangian method since the multiplier estimates $u_k$ are chosen
from the bounded set $B$. In practice, one typically chooses $B$ as a (very large) box, and
defines $u_k$ as the projection of $\lambda_k$ onto $B$ in \ref{item:subproblem_ALM}.
Note that without safeguarding, one
obtains the classical augmented Lagrangian method. However, it is well known that the safeguarded
version possesses superior global convergence properties, see \cite{KanzowSteck2017}.
An overview of augmented Lagrangian methods in constrained optimization can be found in
\cite{BirginMartinez2014}.

Let us comment on potential termination criteria for \cref{alg:ALM}.
On the one hand, \cref{alg:ALM} is designed for the computation of M-stationary points of
\eqref{eq:non_Lipschitz_objective} which, at the latest, will become clear
in \cref{cor:ALM_global_convergence}. Thus, one may check approximate validity of these
stationarity conditions in~\ref{item:termination_ALM}. However, if $q$ or $C$ is variationally
challenging, this might be a nontrivial task. On the other hand, at its core, \cref{alg:ALM} is
a penalty method, so it is also reasonable to check approximate feasibility with respect to the
constraints $G(x)\in K$ in~\ref{item:termination_ALM}.

In \cite{ChenGuoLuYe2017}, the authors suggest to solve \eqref{eq:non_Lipschitz_objective},
where all involved spaces are instances of $\R^n$ while the constraints $G(x)\in K$ are replaced
by smooth inequality and equality constraints, with the classical augmented Lagrangian
method. In case where $q$ is not present and $X$ as well as $Y$ are Euclidean spaces,
\cref{alg:ALM} recovers the partial augmented Lagrangian
scheme studied in \cite{JiaKanzowMehlitzWachsmuth2021} where the authors focus on situations where
$C$ is nonconvex and of challenging variational structure. We note that, technically, \cref{alg:ALM}
is also capable of handling this situation. However, it might be difficult to solve the appearing
subproblems \eqref{eq:ALM_subproblem} if both $q$ and $C$ are variationally complex.
Note that we did not specify in~\ref{item:subproblem_ALM} how precisely the subproblems have to
be solved. Exemplary, one could aim to find stationary or globally $\varepsilon$-minimal points of the function
$\LL_{\theta_k}(\cdot,u_k)_C$ here. We comment on both situations below.

Our theory from \cref{sec:main_result} can be used to show that \cref{alg:ALM} computes
approximately stationary points of \eqref{eq:non_Lipschitz_objective} when the subproblems \eqref{eq:ALM_subproblem}
are solved up to stationarity of $\LL_{\theta_k}(\cdot,u_k)_C$.
\begin{theorem}\label{thm:ALM_produces_asymptotically_stationary_points}
	Let $\{x_k\}_{k\in\N}$ be a sequence generated by \cref{alg:ALM} such that
	$x_{k+1}$ is a stationary point of $\LL_{\theta_k}(\cdot,u_k)_C$ for each $k\in\N$.
	Assume that, along a subsequence (without relabeling), we have
	$x_k\to\bar x$ and $q(x_k)\to q(\bar x)$ for some $\bar x\in X$ which is feasible to \eqref{eq:non_Lipschitz_objective}.
	Then $\bar x$ is an approximately stationary point of \eqref{eq:non_Lipschitz_objective}.
\end{theorem}
\begin{proof}
	Observe that \cref{ass:ALM} guarantees that $\LL_{\theta_k}(\cdot,u_k)$ is lower
	semicontinuous relative to $C$ near each point from $C\cap\dom q$, see
	\cref{cor:sufficient_criteria_lower_semicontinuity_relative_to_a_set,cor:lower_semicontinuity_near_point_set_via_weak_sequential_lower_semicontinuity}.
	Since $x_{k+1}$ is a stationary point of $\LL_{\theta_k}(\cdot,u_k)_C$, we can apply
	\cref{rem:minimality_vs_stationarity} and \cref{thm:main_result_wrt_set} in order to find
	$x_{k+1}'\in B_{1/k}(x_{k+1})$ and $x_{k+1}''\in C\cap B_{1/k}(x_{k+1})$ such that $|q(x_{k+1}')-q(x_{k+1})|<\tfrac1k$ and
	\[
		0\in\partial \LL_{\theta_{k}}(x_{k+1}',u_k)+N_C(x_{k+1}'')+\tfrac1k\,\mathbb B^*
	\]
	for each $k\in\N$.
	From $x_k\to\bar x$ and $q(x_k)\to q(\bar x)$ we have $x_k'\to \bar x$, $x_k''\to\bar x$, and $q(x_k')\to q(\bar x)$.	
	Noting that $f$, $G$, and, by convexity of $K$, the squared distance function $\dist_K^2$
	are continuously differentiable, we find
	\begin{equation}\label{eq:non_Lipschitz_asymptotic_stationarity}
		\begin{aligned}
		0\in f'(x_{k+1}')
		&+
		\theta_k\,G'(x_{k+1}')^*
		\left[
			G(x_{k+1}')+u_k/\theta_k-P_K\left(G(x_{k+1}')+u_k/\theta_k\right)
		\right]
		\\
		&+
		\partial q(x_{k+1}')
		+
		N_C(x_{k+1}'')
		+
		\tfrac1k\,\mathbb B^*
		\end{aligned}
	\end{equation}
	for each $k\in\N$ where we used
	the subdifferential sum rule from \cite[Corollary~1.12.2]{Kru03}.
	Let us set $y_{k+1}:=G(x_{k+1}')-P_K(G(x_{k+1}')+u_k/\theta_k)$
	for each $k\in\N$. By definition of the projection and convexity of $K$, we find
	\begin{align*}
		\theta_k(y_k+u_k/\theta_k)
		\in
		N_K(P_K(G(x_{k+1}')+u_k/\theta_k))
		=
		N_K(G(x_{k+1}')-y_{k+1}),
	\end{align*}
	so we can rewrite \eqref{eq:non_Lipschitz_asymptotic_stationarity} by means of
	\begin{equation}\label{eq:non_Lipschitz_asymptotic_stationarity_refined}
		0\in f'(x_{k+1}')+\partial q(x_{k+1}')+ G'(x_{k+1}')^*N_K(G(x_{k+1}')-y_{k+1})
			+N_C(x_{k+1}'')+\tfrac1k\,\mathbb B^*
	\end{equation}
	for each $k\in\N$.
	
	It remains to show $y_{k+1}\to 0$.
	We distinguish two cases.
	
	First, assume that $\{\theta_k\}_{k\in\N}$ remains bounded.
	By construction of \cref{alg:ALM}, this yields $V_{\theta_k}(x_{k+1},u_k)\to 0$ as $k\to+\infty$.
	Recalling that the projection $P_K$ is Lipschitz continuous with modulus $1$ by convexity
	of $K$, we have
	\begin{align*}
		\norm{y_{k+1}}
		&
		\leq
		V_{\theta_k}(x_{k+1},u_k)
		+
		\norm{G(x_{k+1}')-G(x_{k+1})}\\
		&\qquad
		+
		\norm{P_K(G(x_{k+1}')+u_k/\theta_k)-P_K(G(x_{k+1})+u_k/\theta_k)}
		\\
		&
		\leq
		V_{\theta_k}(x_{k+1},u_k)
		+
		2\norm{G(x_{k+1}')-G(x_{k+1})}
	\end{align*}
	for each $k\in\N$. Due to $x_k\to \bar x$ and $x_k'\to\bar x$ as well as continuity of $G$,
	this yields $y_{k+1}\to 0$.
	
	Finally, suppose that $\{\theta_k\}_{k\in\N}$ is unbounded. Since this sequence is monotonically
	increasing, we have $\theta_k\to+\infty$.
	By boundedness of $\{u_k\}_{k\in\N}$, continuity of $G$ as well as the projection $P_K$,
	$x_k'\to\bar x$, and feasibility of $\bar x$ for \eqref{eq:non_Lipschitz_objective}, it holds
	\[
		y_{k+1}
		=
		G(x_{k+1}')-P_K(G(x_{k+1}')+u_k/\theta_k)
		\to
		G(\bar x)-P_K(G(\bar x))
		=
		0,
	\]
	and this completes the proof.
\end{proof}

Let us mention that the assumption $q(x_k)\to q(\bar x)$ is trivially satisfied as soon as
$q$ is continuous on its domain. For other types of discontinuity, however, this does not follow
by construction of the method and has to be presumed. Let us note that this convergence is also implicitly
used in the proof of the related result \cite[Theorem~3.1]{ChenGuoLuYe2017} but does not follow
from the postulated assumptions, i.e., this assumption is missing there.

Note that demanding feasibility of accumulation points is a natural assumption when considering
augmented Lagrangian methods. This property naturally holds whenever the sequence $\{\theta_k\}_{k\in\N}$
remains bounded or if $q$ is bounded from below while the sequence $\{\LL_{\theta_k}(x_{k+1},u_k)\}_{k\in\N}$ remains bounded.
The latter assumption is typically satisfied whenever globally $\eps_k$-minimal points of $\LL_{\theta_k}(\cdot,u_k)_C$
can be computed in order to approximately solve the
subproblems \eqref{eq:ALM_subproblem} in~\ref{item:subproblem_ALM}, where
$\{\varepsilon_k\}_{k\in\N}\subset[0,+\infty)$ is a bounded sequence. Indeed, we have
\begin{equation}\label{eq:consequence_of_eps_minimality}
	\forall x\in\mathcal S\colon\quad
	\LL_{\theta_k}(x_{k+1},u_k)
	\leq
	\LL_{\theta_k}(x,u_k)+\varepsilon_k
	\leq
	f(x)+\norm{u_k}^2/(2\theta_k)+q(x)+\varepsilon_k
\end{equation}
in this situation, and this yields the claim by boundedness of $\{u_k\}_{k\in\N}$ and monotonicity
of $\{\theta_k\}_{k\in\N}$. If $\{\varepsilon_k\}_{k\in\N}$ is a null sequence, we obtain an even
stronger result.

\begin{theorem}\label{thm:global_convergence_eps_minimality}
	Let $\{x_k\}_{k\in\N}\subset X$ be a sequence generated by \cref{alg:ALM} and let $\{\varepsilon_k\}_{k\in\N}\subset[0,+\infty)$
	be a null sequence such that $x_{k+1}$ is a globally $\varepsilon_k$-minimal point of $\LL_{\theta_k}(\cdot,u_k)_C$ for each $k\in\N$.
	Then each accumulation point $\bar x\in X$ of $\{x_k\}_{k\in\N}$ is a global minimizer of
	\eqref{eq:non_Lipschitz_objective} and, along the associated subsequence, we find $q(x_k)\to q(\bar x)$.
\end{theorem}
\begin{proof}
	Without loss of generality, we assume $x_k\to\bar x$.
	By closedness of $C$, we have $\bar x\in C$.
	The estimate \eqref{eq:consequence_of_eps_minimality} yields
	\begin{equation}\label{eq:consequence_of_eps_minimality_rearranged}
		f(x_{k+1})+q(x_{k+1})+\frac{\theta_k}{2}\dist_K^2\left(G(x_{k+1})+\frac{u_k}{\theta_k}\right)-\frac{\norm{u_k}^2}{2\theta_k}
		\leq
		f(x)+q(x)+\varepsilon_k
	\end{equation}
	for each $x\in\mathcal S$.
	We show the statement of the theorem by distinguishing two cases.
	
	In case where $\{\theta_k\}_{k\in\N}$ remains bounded, we find $\dist_K(G(x_{k+1}))\leq V_{\theta_k}(x_{k+1},u_k)\to 0$
	from~\ref{item:control_of_parameter_ALM}, so the continuity of the distance function $\dist_K$ and $G$ yields $G(\bar x)\in K$, i.e.,
	$\bar x$ is feasible to \eqref{eq:non_Lipschitz_objective}. Using the triangle inequality, we also obtain
	\[
		\dist_K(G(x_{k+1})+u_k/\theta_k)
		\leq
		\dist_K(G(x_{k+1}))+\norm{u_k}/\theta_k
		\leq
		V_{\theta_k}(x_{k+1},u_k)+\norm{u_k}/\theta_k
	\]
	for each $k\in\N$. Squaring on both sides, exploiting the boundedness of $\{u_k\}_{k\in\N}$ and $V_{\theta_k}(x_{k+1},u_k)\to 0$ yields
	\[
		\limsup\limits_{k\to+\infty}\left(\dist_K^2\left(G(x_{k+1})+u_k/\theta_k\right)-(\norm{u_k}/\theta_k)^2\right)\leq 0.
	\]
	The boundedness of $\{\theta_k\}_{k\in\N}$ and \eqref{eq:consequence_of_eps_minimality_rearranged} thus show
	$\limsup_{k\to+\infty}(f(x_{k+1})+q(x_{k+1}))\leq f(x)+q(x)$ for each $x\in\mathcal S$.
	Exploiting the lower semicontinuity of $q$, this leads to $f(\bar x)+q(\bar x)\leq f(x)+q(x)$, i.e., $\bar x$ is a global minimizer
	of \eqref{eq:non_Lipschitz_objective}. On the other hand, we have
	\[
		f(\bar x)+q(\bar x)
		\leq
		\liminf\limits_{k\to+\infty}\left(f(x_{k+1})+q(x_{k+1})\right)
		\leq
		\limsup\limits_{k\to+\infty}\left(f(x_{k+1})+q(x_{k+1})\right)
		\leq
		f(\bar x)+q(\bar x)
	\]
	from the particular choice $x:=\bar x$, so the continuity of $f$ yields $q(x_k)\to q(\bar x)$ as claimed.
	
	Now, let us assume that $\{\theta_k\}_{k\in\N}$ is not bounded. Then we have $\theta_k\to+\infty$ from~\ref{item:control_of_parameter_ALM}.
	By choice of $x_{k+1}$, we have $\LL_{\theta_k}(x_{k+1},u_k)\leq \LL_{\theta_k}(x,u_k)+\varepsilon_k$ for all $x\in C$ and each $k\in\N$,
	so the definition of the augmented Lagrangian function yields
	\[
		f(x_{k+1})+q(x_{k+1})+\frac{\theta_k}{2}\dist_K^2\left(G(x_{k+1})+\frac{u_k}{\theta_k}\right)
		\leq
		f(x)+q(x)+\frac{\theta_k}{2}\dist^2_K\left(G(x)+\frac{u_k}{\theta_k}\right)+\varepsilon_k
	\]
	for each $x\in C$. By continuity of $f$ and lower semicontinuity of $q$, $\{f(x_{k+1})+q(x_{k+1})\}_{k\in\N}$ is bounded from below.
	Thus, dividing the above estimate by $\theta_k$ and taking the limit inferior, we find
	\begin{align*}
		\dist_K^2(G(\bar x))
		&=
		\liminf\limits_{k\to+\infty} \dist_K^2\left(G(x_{k+1})+u_k/\theta_k\right)\\
		&\leq
		\liminf\limits_{k\to+\infty} \dist_K^2\left(G(x)+u_k/\theta_k\right)
		=
		\dist_K^2(G(x))
	\end{align*}
 	for each $x\in C$ from $\theta_k\to+\infty$ and continuity of $\dist_K$ and $G$. Hence, $\bar x$ is a global minimizer of $\dist_K^2\circ G$
 	over $C$. Since $\mathcal S$ is assumed to be nonempty, we infer $\dist_K^2(G(\bar x))=0$, i.e., $\bar x$ is feasible to
 	\eqref{eq:non_Lipschitz_objective}.
 	Exploiting boundedness of $\{u_k\}_{k\in\N}$, nonnegativity of the distance function, and $\theta_k\to+\infty$, we now obtain
 	$\limsup_{k\to+\infty}(f(x_{k+1})+q(x_{k+1}))\leq f(x)+q(x)$ for each $x\in\mathcal S$ from
 	\eqref{eq:consequence_of_eps_minimality_rearranged}.
 	Proceeding as in the first case now yields the claim.
\end{proof}

It remains to clarify how the subproblems \eqref{eq:ALM_subproblem} can be solved in practice.
If the non-Lipschitzness of $q$ is, in some sense, structured while $C$ is of simple form, it should be
reasonable to solve \eqref{eq:ALM_subproblem} with the aid of a nonmonotone proximal gradient method,
see \cite[Section~3.1]{ChenGuoLuYe2017}.
On the other hand, in situations where $q$ is not present while $C$ possesses a variational
structure which allows for the efficient computation of projections, a nonmonotone spectral gradient method might be used to
solve \eqref{eq:ALM_subproblem}, see \cite[Section~3]{JiaKanzowMehlitzWachsmuth2021}.
Finally, it might be even possible to solve \eqref{eq:ALM_subproblem} up to global optimality
in analytic way in some practically relevant applications where $q$ is a standard sparsity-promoting
term and the remaining data is simple enough.

Coming back to the assertion of \cref{thm:ALM_produces_asymptotically_stationary_points}, the following is now clear from \cref{cor:asymptotic_regularity_CQ}.
\begin{corollary}\label{cor:ALM_global_convergence}
	Let $\{x_k\}_{k\in\N}$ be a sequence generated by \cref{alg:ALM} such that $x_{k+1}$ is a stationary point
	of $\LL_{\theta_k}(\cdot,u_k)_C$ for each $k\in\N$.
	Assume that, along a subsequence (without relabeling), we have
	$x_k\to\bar x$ and $q(x_k)\to q(\bar x)$ for some $\bar x\in X$ which is feasible to \eqref{eq:non_Lipschitz_objective}
	and satisfies the uniform qualification condition.
	Then $\bar x$ is M-stationary.
\end{corollary}

Note that in the light of \cref{lem:sufficient_condition_sequential_regularity},
\cref{cor:ALM_global_convergence} drastically generalizes and improves
\cite[Theorem~3.1]{ChenGuoLuYe2017} which shows global convergence of a related augmented
Lagrangian method to certain stationary points under validity of a basic qualification,
see condition \eqref{eq:BCQ}, and the
\emph{relaxed constant positive linear dependence constraint qualification} which is more restrictive
than condition \eqref{eq:uniform_CQ}
in the investigated setting, see \cite[Lemma~2.7]{JiaKanzowMehlitzWachsmuth2021} as well.
Let us mention that such a result has been foreshadowed in \cite[Section~5.4]{JiaKanzowMehlitzWachsmuth2021}.
We would like to point the reader's attention to the fact that working with strong
accumulation points in the context of \cref{thm:ALM_produces_asymptotically_stationary_points,thm:global_convergence_eps_minimality} and
\cref{cor:ALM_global_convergence} is indispensable as long as $q$ or the sets $K$ and $C$ are not
convex since the limiting variational tools rely on strong convergence in the primal space.
In the absence of $q$ and if $K$ and $C$ are convex, some convergence results based on weak
accumulation points are available, see e.g.\ \cite[Section~7]{BoergensKanzowMehlitzWachsmuth2019} and \cite{BoergensKanzowSteck2019,KanzowSteckWachsmuth2018}.
Clearly, in finite dimensions, both types of convergence are equivalent and the consideration
of strong accumulation points is not restrictive at all.

\subsection{Sparsity-promotion in optimal control}\label{sec:control}

In this section, we apply the theory derived earlier to an optimal control problem
with a sparsity-promoting term in the objective function.
As it is common to denote control functions by $u$ in the context of optimal control,
we will use the same notation here for the decision variable for notational convenience.

For some bounded domain $D\subset\R^d$ and some $p\in(0,1)$, we define a function
$q\colon L^2(D)\to\R$ by means of
\begin{equation}\label{eq:sparsity_promoting_functional}
	\forall u\in L^2(D)\colon
	\quad
	q(u):=\int_D |u(\omega)|^p\,\mathrm d\omega.
\end{equation}
Above, $L^2(D)$ denotes the standard Lebesgue space of (equivalence classes of)
measurable functions whose square is integrable and is equipped with the usual norm.
In optimal control, the function $q$ is used as an additive term in the objective function
in order to promote sparsity of underlying control
functions, see \cite{ItoKunisch2014,NatemeyerWachsmuth2020,Wachsmuth2019}.
A reason for this behavior is that the integrand $t\mapsto |t|^p$ possesses a unique global minimizer
and infinite growth at the origin.
In \cite{MehlitzWachsmuth2021}, the authors explore the variational properties of the
functional $q$. It has been shown to be uniformly continuous in \cite[Lemma~2.3]{MehlitzWachsmuth2021}.
Furthermore, in \cite[Theorem~4.6]{MehlitzWachsmuth2021}, the following formula has been proven
for each $\bar u\in L^2(D)$:
\begin{equation}\label{eq:subdifferentials_sparsity}
	\bsd q(\bar u)=\sd q(\bar u)
	=
	\bigl\{
		\eta\in L^2(D)\,|\,
		\eta= p\abs{\bar u}^{p-2}\bar u\text{ a.e.\ on }\{\bar u\neq 0\}
	\bigr\}.
\end{equation}
Let us emphasize that this means that the Fr\'{e}chet and limiting subdifferential actually coincide
and can be empty if the reference point is a function which tends to zero too fast somewhere on its
domain. This underlines the sparsity-promoting properties of $q$.

Now, for a continuously differentiable function $f\colon L^2(D)\to\R$ and functions
$u_a,u_b\in L^2(D)$ satisfying $u_a<0<u_b$ almost everywhere on $D$,
we consider the optimization problem
\begin{equation}\label{eq:optimal_control}\tag{OC}
	\min\limits_u\{f(u)+q(u)\,|\,u\in C\}
\end{equation}
where $C\subset L^2(D)$ is given by the box
\[
	C:=\{u\in L^2(D)\,|\,u_a\leq u\leq u_b\text{ a.e.\ on }D\}.
\]
For later use, let us mention that, for each $u\in C$, the (Fr\'{e}chet) normal cone to $C$ at $u$ is
given by the pointwise representation
\begin{equation}\label{eq:normal_cone_to_pointwise_box}
	N_C(u)
	=
	\left\{
		\eta\in L^2(D)\,\middle|\,
		\begin{aligned}
			&\eta\leq 0&&\text{a.e.\ on $\{u<u_b\}$}\\
			&\eta\geq 0&&\text{a.e.\ on $\{u_a<u\}$}
		\end{aligned}
	\right\}.
\end{equation}
Typically, in optimal control, $f$ is a function of type
\begin{equation}\label{eq:target_type_objective}
	\forall u\in L^2(D)\colon\quad
	f(u):=\tfrac12\norm{S(u)-y_\textup{d}}^2+\tfrac{\sigma}{2}\norm{u}^2
\end{equation}
where $S\colon L^2(D)\to H$ is the continuously differentiable
control-to-observation operator associated with a given
system of differential equations, $H$ is a Hilbert space, $y_\textup{d}\in H$ is the desired state, and $\sigma\geq 0$ is a
regularization parameter. Clearly, by means of the chain rule, $f$ is continuously
differentiable with derivative given by
\[
	\forall u\in L^2(D)\colon\quad
	f'(u)=S'(u)^*[S(u)-y_\textup{d}]+\sigma u.
\]
The presence of $q$ in the objective functional of
\eqref{eq:optimal_control} enforces sparsity of its solutions, i.e., the support of optimal controls
is likely to be small. It already has been mentioned in \cite{ItoKunisch2014,NatemeyerWachsmuth2020}
that one generally cannot show existence of solutions to optimization problems of type
\eqref{eq:optimal_control}. Nevertheless, the practical need for sparse controls makes it attractive
to consider the model and to derive necessary optimality conditions in order to identify reasonable
stationary points.

In the subsequent lemma, we show that the feasible points of \eqref{eq:optimal_control} satisfy
the uniform qualification condition stated in \cref{def:asymptotic_regularity}.
\begin{lemma}\label{lem:asymptotic_regularity_OC}
	Let $\bar u\in L^2(D)$ be a feasible point of \eqref{eq:optimal_control}.
	Then the uniform qualification condition holds at $\bar u$.
\end{lemma}
\begin{proof}
	Recalling that $q$ is continuous while $C$ is convex,
	the uniform qualification condition takes the simplified form
	\[
		\limsup\limits_{u\to\bar u,\,u'\to\bar u}
		\bigl(\sd q(u)+N_C(u')\bigr)
		\subset
		\bsd q(\bar u)+N_C(\bar u).
	\]
	Let us fix some point $\eta\in\limsup_{u\to\bar u,\,u'\to\bar u}\bigl(\sd q(u)+N_C(u')\bigr)$.
	Then we find sequences
	$\{u_k\}_{k\in\N},\{u_k'\}_{k\in\N},\{\eta_k\}_{k\in\N}\subset L^2(D)$
	such that $u_k\to\bar u$, $u_k'\to\bar u$, $\eta_k\to \eta$, as well as
	$\eta_k\in\sd q(u_k)+N_C(u_k')$ for all $k\in\N$.
	Particularly, there are sequences $\{\xi_k\}_{k\in\N},\{\mu_k\}_{k\in\N}\subset L^2(D)$
	such that $\xi_k\in\sd q(u_k)$, $\mu_k\in N_C(u_k')$, and $\eta_k=\xi_k+\mu_k$ for all $k\in\N$.
	From \eqref{eq:subdifferentials_sparsity} we find $\xi_k=p\abs{u_k}^{p-2}u_k$ almost
	everywhere on $\{u_k\neq 0\}$ for each $k\in\N$. Furthermore, we have $\mu_k\leq 0$
	almost everywhere on $\{u_k'=u_a\}$, $\mu_k\geq 0$ almost everywhere on $\{u_k'=u_b\}$,
	and $\mu_k=0$ almost everywhere on $\{u_a<u_k'<u_b\}$ for each $k\in\N$
	from \eqref{eq:normal_cone_to_pointwise_box}.
	Along a subsequence (without relabeling) we can ensure the convergences 
	$u_k(\omega)\to\bar u(\omega)$, $u_k'(\omega)\to\bar u(\omega)$, and $\eta_k(\omega)\to \eta(\omega)$
	for almost every $\omega\in D$.
	Thus, for almost every $\omega\in\{\bar u=u_a\}$, we can guarantee $u_k(\omega)<0$ and $u_k'(\omega)\in[u_a(\omega),0)$,
	i.e., $\eta_k(\omega)=\xi_k(\omega)+\mu_k(\omega)\leq p|u_k(\omega)|^{p-2}u_k(\omega)$
	for all large enough $k\in\N$, so, taking the
	limit yields $\eta(\omega)\leq p\abs{\bar u(\omega)}^{p-2}\bar u(\omega)$.
	Similarly, we find $\eta(\omega)\geq p\abs{\bar u(\omega)}^{p-2}\bar u(\omega)$ for almost every
	$\omega\in\{\bar u=u_b\}$. Finally, for almost every $\omega\in\{\bar u\neq 0\}\cap\{u_a<\bar u<u_b\}$,
	we have $u_k(\omega)\neq 0$ and $u_a(\omega)<u_k'(\omega)<u_b(\omega)$, i.e.,
	$\eta_k(\omega)=p\abs{u_k(\omega)}^{p-2}u_k(\omega)$ for large enough $k\in\N$, so taking
	the limit, we have $\eta(\omega)=p\abs{\bar u(\omega)}^{p-2}\bar u(\omega)$.
	Again, from \eqref{eq:subdifferentials_sparsity} and \eqref{eq:normal_cone_to_pointwise_box},
	we have $\eta\in\bsd q(\bar u)+N_C(\bar u)$,
	and this yields the claim.
\end{proof}

Recalling that $q$ is uniformly continuous, the subsequent result now directly follows
from \cref{cor:asymptotic_regularity_CQ}, the above lemma, and formulas
\eqref{eq:subdifferentials_sparsity} as well as \eqref{eq:normal_cone_to_pointwise_box}.
\begin{theorem}\label{thm:optimality_conditions_sparse_control}
	Let $\bar u\in L^2(D)$ be a local minimizer of \eqref{eq:optimal_control}.
	Then there exists a function $\eta\in L^2(D)$ such that
	\begin{subequations}\label{eq:sparse_control}
		\begin{align}
			\label{eq:sparse_control_der}
			&f'(\bar u)+\eta=0,\\
			\label{eq:sparse_control_subgradient_q}
			&\eta=p|\bar u|^{p-2}\bar u\quad\text{a.e.\ on }\{\bar u\neq 0\}\cap\{u_a<\bar u<u_b\},\\
			\label{eq:sparse_control_normal_cone_Uad_ua}
			&\eta\leq p\abs{u_a}^{p-2}u_a\quad\text{a.e.\ on }\{\bar u=u_a\},\\
			\label{eq:sparse_control_normal_cone_Uad_ub}
			&\eta\geq p\abs{u_b}^{p-2}u_b\quad\text{a.e.\ on }\{\bar u=u_b\}.
		\end{align}
	\end{subequations}
\end{theorem}

We note that our approach to obtain necessary optimality conditions for \eqref{eq:optimal_control}
is much different from the one used in \cite{ItoKunisch2014,NatemeyerWachsmuth2020} where
\emph{Pontryagin's maximum principle} has been used to derive pointwise conditions characterizing
local minimizers under more restrictive assumptions than we needed to proceed.
On the one hand, this led to optimaility conditions which also provide information on
the subset of $D$ where the locally optimal control is zero, and one can easily see
that this is not the case in \cref{thm:optimality_conditions_sparse_control}.
On the other hand, a detailed inspection of \eqref{eq:sparse_control}
makes clear that our necessary optimality conditions provide helpful information regarding
the structure of the optimal control as the multiplier $\eta$ possesses $L^2$-regularity
while \eqref{eq:sparse_control_subgradient_q} causes $\eta$ to possess singularities
as the optimal control tends to zero somewhere on the domain.
Thus, this condition clearly promotes sparse controls which either are zero, tend to zero (if at all)
slowly enough, or are bounded away from it.
Note that this differs from the conditions derived in \cite{ItoKunisch2014,NatemeyerWachsmuth2020}
which are multiplier-free.

\section{Concluding remarks}\label{sec:conclusions}

In this paper, we established a theory on approximate stationarity conditions for optimization
problems with potentially non-Lipschitzian objective functions in a very general setting.
In contrast to the finite-dimensional situation, where approximate stationarity has been shown to
serve as a necessary optimality condition for local optimality without any additional assumptions,
some additional semicontinuity properties need to be present in the infinite-dimensional context.
We exploited our findings in order to re-address the classical topic of set extremality and were
in position to derive a novel version of the popular extremal principle. This may serve as a
starting point for further research which compares the classical as well as the new version of
the extremal principle in a more detailed way.
Moreover, we used our results in order to derive an approximate notion of stationarity as well
as an associated qualification condition related to M-stationarity for optimization problems with a
composite objective function and geometric constraints
in the Banach space setting. This theory then has been applied to study the convergence properties
of an associated augmented Lagrangian method for the numerical solution of such problems.
Furthermore, we demonstrated how these findings can be used to derive necessary optimality conditions
for optimal control problems with control constraints and a sparsity-promoting term in the
objective function. Some future research may clarify whether our approximate stationarity conditions
can be used to find necessary optimality conditions for optimization problems in function spaces where
nonconvexity or nonsmoothness pop up in a different context.
Exemplary, it would be interesting to study situations where the solution operator $S$
appearing in \eqref{eq:target_type_objective} is nonsmooth, see e.g.\
\cite{ChristofMeyerWalterClason2018,HintermuellerMordukhovichSurowiec2014,RaulsWachsmuth2020},
where the set of feasible controls is nonconvex, see e.g.\
\cite{ClasonRundKunisch2017,ClasonDengMehlitzPruefert2020,MehlitzWachsmuth2018},
or where the function $q$ is a term promoting sharp edges in continuous image denoising or deconvolution,
see e.g.\ \cite[Section~6]{BrediesLorenz2018}.

\section*{Acknowledgments}
The authors are grateful to Hoa Bui who suggested \cref{ex:extremality_more_restrictive}.

This work is supported by the Australian Research Council, project DP160100854, and
the DFG Grant \emph{Bilevel Optimal Control: Theory, Algorithms, and Applications}
(Grant No.\ WA 3636/4-2) within the Priority Program SPP 1962 (Non-smooth
and Complementarity-based Distributed Parameter Systems: Simulation and Hierarchical Optimization).
The first author
benefited from the support of the European Union's Horizon 2020
research and innovation programme under the Marie Sk{\l}odowska--Curie
Grant Agreement No.\ 823731 CONMECH, and Conicyt REDES program 180032.


\end{document}